\documentclass[1p]{elsarticle}
\usepackage{lineno,hyperref}
\journal{Journal of Functional Analysis}

\usepackage{amsmath}

\newcommand{\Rex}{\overline{\mathbb{R}}}

\usepackage{amssymb}

\usepackage{enumitem}

\usepackage[capitalise]{cleveref}

\newtheorem{theorem}{Theorem}
\newtheorem{lemma}[theorem]{Lemma}
\newtheorem{corollary}[theorem]{Corollary}
\newtheorem{proposition}[theorem]{Proposition}

\newdefinition{definition}{Definition}
\newproof{proof}{Proof}
\newdefinition{remark}{Remark}
\newdefinition{example}{Example}
\newproof{pot}{Proof of Theorem \ref{thm2}}

\let\epsilon\varepsilon
\DeclareMathOperator{\cl}{cl}

\DeclareMathOperator{\co}{co}
\DeclareMathOperator{\lin}{lin}
\DeclareMathOperator{\cco}{\overline{co}}
\DeclareMathOperator{\aff}{aff}

\DeclareMathOperator{\epi}{epi}

\DeclareMathOperator{\grafo}{gph}
\DeclareMathOperator{\dom}{dom}

\DeclareMathOperator{\inte}{int}
\DeclareMathOperator{\spn}{span}

\DeclareMathOperator{\sub}{\partial}
\usepackage{dsfont}

\begin{document}

\begin{frontmatter}

\title{Characterizations of the subdifferential of convex integral
	functions under qualification conditions}
\tnotetext[mytitlenote]{This work is partially supported by CONICYT grants: Fondecyt 1151003, Fondecyt 1190012, Fondecyt 1150909, Basal PFB -03 and Basal FB0003, CONICYT-PCHA/doctorado Nacional / 2014-21140621.}

\author{Rafael Correa\corref{CORREA}}
\address{Universidad de O'Higgins and DIM-CMM Universidad de Chile, Chile}
\ead{rcorrea@dim.uchile.cl}

\author{Abderrahim Hantoute\corref{HANTAOUTE}}
\address{Center for Mathematical Modeling, Universidad de Chile, Chile}
\ead{ahantoute@dim.uchile.cl}
\author{Pedro P\'erez-Aros\corref{PEREZ}}
\address{Instituto de Ciencias de la Ingenier\'ia, Universidad de O'Higgins, Chile}
\ead{pedro.perez@uoh.cl}

\begin{abstract}
This work provides formulae for the $\epsilon$-subdifferential of integral functions in the framework of complete $\sigma$-finite measure spaces and locally convex spaces. In this work we  present here new formulae for this $\epsilon$-subdifferential under the presence of continuity-type qualification conditions relying on the data involved in the integrand. 
\end{abstract}


\end{frontmatter}

We provide new formulae for the subdifferential and the $\epsilon$-subdifferential of the convex integral function given by the following expression
\begin{equation*}
	I_f (x):=\int\limits_T f(t,x) d\mu(t)
\end{equation*}
where $(T,\Sigma,\mu)$ is a complete $\sigma$-finite measure space, and $f:T\times X\to \Rex$ is a convex normal integrand defined on a locally convex space $X$.

General formulae have been established in \cite{INTECONV} using a finite-dimentional reduction approach, without additional assumptions on the data represented by the integrand $f$. In this paper, we use natural qualifications condition, involving appropriate continuity assumption on the integrand, to give more explicit characterization of the subdifferential and the $\epsilon$-subdifferential of the function $I_f$. This research is in the line of the theory of convex analysis and duality, which provides a rich literature for subdifferential calculus of convex functions (see e.g. \cite{MR1245600,MR0467080,MR3522006,MR3561780,MR1721727,MR2346533,MR2144010,MR2986672,MR1491362,MR1451876,MR0373611}).

The rest of the paper is organized as follows: In \Cref{Notation} we give  standard notations and definitions of convex analysis. In \Cref{Pre} we give main definitions and notations of measure theory, vector-valued integration, integrand functions, integral functionals and integration of set valued-maps. In \Cref{QUA} we characterize the subdifferential of $I_f$ under classical continuity-type qualification conditions that involve the data (see \cref{corollaryAsplund,Exactintegralformula}). In \Cref{CONJ} we give a characterization for the $\epsilon$-normal sets, in terms of the functions $f_t's$ (see Proposition \ref{lemma:normalcone}). Finally, we give in \Cref{BronstedRockafellartheorems} other sequential characterizations for the subdifferential of $I_f$.

\section{Notation}

\label{Notation}

Throughout the paper, the symbols $(X,\tau _{X})$ and $(X^{\ast },\tau _{X^{\ast
}})$ denote two  (Hausdorff) locally convex spaces (lcs, for short). The associated (symmetric) bilinear form is $\langle \cdot ,\cdot \rangle :X^{\ast
}\times X\rightarrow \mathbb{R}$, $\langle x^{\ast },x\rangle =\langle
x,x^{\ast }\rangle =x^{\ast }(x)$. For a point $x\in X$ ($x^{\ast }\in
X^{\ast }$, resp.) $\mathcal{N}_{x}(\tau _{X})$ ($\mathcal{N}_{x^{\ast
}}(\tau _{X^{\ast }})$, resp.) represents the (convex, balanced and symmetric)
neighborhoods system of $x$ ($x^{\ast }$, resp.) with respect to the topology 
$\tau _{X}$ ($\tau _{X^{\ast }}$, resp.); we omit the reference to the
topology when there is no confusion. Examples of $\tau _{X^{\ast }}$ are the
weak$^{\ast }\ $topology $w(X^{\ast },X)$ ($w^{\ast },$ for short), the Mackey topology denoted by 
$\tau (X^{\ast },X)$, and the strong topology denoted by $\beta (X^{\ast
},X).$ The real extended line is denoted by $\overline{\mathbb{R}}:=\mathbb{R}\cup \{-\infty
,\infty \}$ and we adopt the conventions that $0\cdot \infty =0=0\cdot (-\infty
)$ and $\infty +(-\infty )=(-\infty )+\infty =\infty .$ We denote $B_{\rho
}(x,r):=\{y\in X:\rho (x-z)\leq r\}$ if $\rho :X\rightarrow \mathbb{R}$ is a
seminorm, $x\in X,$ and $r>0$.

Let $f:X\rightarrow \overline{\mathbb{R}}$ be a function, the 
\emph{domain} of $f$ is $\dom f:=\{x\in X\mid f(x)<+\infty \}$. The function $f$ is said to be \emph{proper} if $\dom f\neq \emptyset $ and $%
f>-\infty $. The \emph{conjugate} of $f$ is
the function $f^{\ast }:X^{\ast }\rightarrow \overline{\mathbb{R}}$ defined
by 
\begin{equation*}
	f^{\ast }(x^{\ast }):=\sup_{x\in X}\{\langle x^{\ast },x\rangle -f(x)\},
\end{equation*}%
and the \emph{biconjugate} of $f$ is $f^{\ast \ast }:=(f^{\ast })^{\ast
}:X\rightarrow \overline{\mathbb{R}}.$ For $\varepsilon \geq 0$ the $%
\varepsilon $-\emph{subdifferential} of $f$ at a point $x\in X$ where it is
finite is the set 
\begin{equation*}
	\partial _{\varepsilon }f(x):=\{x^{\ast }\in X^{\ast }\mid \langle x^{\ast
	},y-x\rangle \leq f(y)-f(x)+\varepsilon ,\ \forall y\in X\};
\end{equation*}%
if $f(x)$ is not finite, we set $\partial _{\varepsilon }f(x):=\emptyset $.

The \emph{indicator} and the \emph{support} functions of a set $A$ ($%
\subseteq X,X^{\ast }$) are, respectively, 
\begin{equation*}
	\delta _{A}(x):=%
	\begin{cases}
		0\qquad & x\in A \\ 
		+\infty & x\notin A,%
	\end{cases}%
	\qquad \qquad \sigma _{A}:=\delta _{A}^{\ast }.
\end{equation*}%

For a set $A\subseteq X$, we denote by $\inte(A)$, $\overline{A}$ (or $\cl A $), $\co(A)$, 
$\cco(A)$, $\lin(A)$ and $\aff(A)$, the \textit{interior}, the \textit{%
	closure}, the \emph{convex hull}, the \emph{closed convex hull}, the %
\emph{linear subspace} and the \emph{affine subspace} of $A$.  The \emph{polar} of $A$ is
the set 
\begin{equation*}
	A^{o}:=\{x^{\ast }\in X^{\ast }\mid \langle x^{\ast },x\rangle \leq
	1,\forall x\in A\},
\end{equation*}%
and the \emph{recession cone} of $A$ (when $A$ is convex) is the set 
\begin{equation*}
	A_{\infty }:=\{x\in X\mid \lambda x+y\in A\text{ for some }y\text{ in }A%
	\text{ and all }\lambda \geq 0\}.
\end{equation*}%
The $\varepsilon $-normal set of $A$ at $x$ is $\emph{N}_{A}^{\varepsilon
}(x):=\partial _{\varepsilon }\delta _{A}(x)$.

\section{Preliminary results}\label{Pre}

In what follows $(X,\tau _{X})$ and $(X^{\ast },\tau _{X^{\ast }})$ are both
lcs, as in  \Cref{Notation}. 
We give the main definitions and results
which are used in the sequel.

First, we introduce the following class of functions, referred to as the
class of epi-pointed functions, which has been shown to be useful for many
purposes (see, for instance, \cite{PEREZ2018,BronstedRock,PerezAros2018,MR2960092,MR1884906,MR2861329,MR2739586,COHASA,MR3507100,MR3033098}. As far as we know the definition used in this work was first introduced in \cite{COHASA} with the name of \emph{Mackey-Epipointed function}.

\begin{definition} A\ function $f:X\rightarrow 
	\overline{\mathbb{R}}$ is said to be epi-pointed if $f^{\ast }$ is proper
	and $\tau (X^{\ast },X)$-continuous at some point of its domain.
\end{definition}

A Hausdorff topological space $S$ is said to be a Suslin space if there
exist a Polish space $P$ (complete, metrizable and separable) and a
continuous surjection from $P$ to $S$ (see \cite%
{MR0027138,MR0467310,MR0426084}). For example, if $X$ is a separable Banach
space, then $(X,\left\Vert \cdot \right\Vert )$ and $(X^{\ast },w^{\ast })$
are Suslin.

Let $(T,\Sigma ,\mu )$ be a complete $\sigma $-finite measure space. Given a
function $f:T\rightarrow \overline{\mathbb{R}}$, we denote 
$\mathcal{D}_{f}:=\{g\in L^{1}(T,\mathbb{R}):f(t)\leq g(t)\;\mu \text{-almost
	everywhere}\},$ and define the upper integral of $f$ by 
\begin{equation}
	\int\limits_{T}f(t)d\mu (t):=\inf_{g\in \mathcal{D}_{f}}\int\limits_{T}g(t)%
	\mu (t)  \label{defint}
\end{equation}%
whenever $\mathcal{D}_{f}\neq \emptyset $. If $\mathcal{D}_{f}=\emptyset $,
we set $\int\limits_{T}f(t)d\mu (t):=+\infty $.
A function $f:T\rightarrow U,$ with $U$ being a topological space, is called
simple if there are $k\in \mathbb{N}$, a partition $T_{i}\in \Sigma $ and
elements $x_{i}\in U$, $i=0,...,k,$ such that $f=\sum_{i=0}^{k}x_{i}%
\mathds{1}_{T_{i}}$ (here, $\mathds{1}_{T_{i}}$ denotes the characteristic
function of $T_{i}$, equal to $1$ in $T_{i}$ and $0$ outside). Function $%
f\ $is called strongly measurable (measurable, for short) if there exists a
countable family $(f_{n})_n$ of simple functions such that $%
f(t)=\lim\limits_{n \rightarrow \infty }f_{n}(t)$ for almost every (ae, for
short) $t\in T$.

A strongly measurable function $f:T\rightarrow X$ is said to be strongly 
integrable (integrable for short), and we write $f\in \mathcal{L}^{1}(T,X)$, if $\int_{T}\sigma
_{B}(f(t))d\mu (t)<\infty $ for every bounded balanced subset $B\subset
X^{\ast }$. Observe that in the Banach space setting, $\mathcal{L}^{1}(T,X)$
is the set of Bochner integrable functions (see, e.g., \cite[\S II]{MR0453964}%
).

A function $f:T\rightarrow X$ is called (weakly or scalarly integrable)
weakly or scalarly measurable if for every $x^{\ast }\in X^{\ast },$ $%
t\rightarrow \langle x^{\ast },f(t)\rangle $ is (integrable, resp.,)
measurable. We denote $\mathcal{L}_{w}^{1}(T,X)$ the space of all weakly
integrable functions {$f$ such that $$\int_{T}\sigma _{B}(f(t))d\mu
	(t)<\infty $$ for every bounded balanced subset $B\subseteq X^{\ast }$}.
Similarly, for functions taking values in $X^*$, we say that $f:T\rightarrow
X^*$ is (w$^*$-integrable, resp.) w$^*$-measurable if for every $x\in X, $
the mapping $t\rightarrow \langle x,f(t)\rangle $ is (integrable, resp.,)
measurable. Also, we denote $\mathcal{L}_{w^*}^{1}(T,X^*)$ the space of all w%
$^*$-integrable functions $f$ such that $\int_{T}\sigma _{B}(f(t))d\mu
(t)<\infty $ for every bounded balanced subset $B\subseteq X$.

It is clear that every strongly integrable function is weakly integrable.
However, the weak measurability of a function $f$ does not necessarily imply
the measurability of the function $\sigma _{B}(f(\cdot ))$, and so the
corresponding integral of this last function must be understood in the sense
of \cref{defint}. Also, observe that if in addition $X$ is a Suslin, then
every $(\Sigma ,\mathcal{B}(X))$-measurable function $f:T\rightarrow X$
(that is, $f^{-1}(B)\in \Sigma $ for all $B\in \mathcal{B}(X)$) is weakly
measurable, where $\mathcal{B}(X)$ is the Borel $\sigma $-Algebra of the
open (equivalently, weak open) set of $X$ (see, e.g., \cite[Theorem III.36 ]{MR0467310}).

The quotient spaces $L^{1}(T,X)$ and $L_{w}^{1}(T,X)$ of $\mathcal{L}%
^{1}(T,X)$ and $\mathcal{L}_{w}^{1}(T,X)$, respectively, are those given
with respect to the equivalence relations $f=g$ ae, and $\langle f,x^{\ast
}\rangle =\langle g,x^{\ast }\rangle $ ae for all $x^{\ast }\in X^{\ast },$
respectively (see, for example, \cite{MR0276438}).

It is worth observing that when $X$ is a separable Banach space, both
notions of (strong and weak) measurability and integrability coincide;
hence, if, in addition, $(X^{\ast },\Vert \cdot \Vert )$ is separable, then $%
\mathcal{L}^{1}(T,X^{\ast })=\mathcal{L}_{w^{\ast }}^{1}(T,X^{\ast })$ (see 
\cite[\S II, Theorem 2]{MR0453964}). It is worth recalling that when the space 
$X$ is separable, but the dual $X^*$ is not $\| \cdot \|$-separable, $%
\mathcal{L}^{1}(T,X^{\ast })$ and $\mathcal{L}_{w^{\ast }}^{1}(T,X^{\ast })$
may not coincide (see \cite[\S II Example 6]{MR0453964}). For every w$^*$%
-integrable function $f:T\rightarrow X^{\ast }$ and every $E\in \Sigma $,
the function $x_{E}^{\sharp }$ defined on $X$ as $x_{E}^{\sharp
}(x):=\int_{E}\langle f,x\rangle d\mu $ is a linear mapping (not necessary
continuous), which we call the weak integral of $f$ over $E$, and we write $%
\int_{E}fd\mu:= x_E^{\sharp }$. Moreover, if $f$ is strongly integrable,
this element $\int_{E}fd\mu$ also refers to the strong integral of $f$ over $%
E$. Observe that, in general, $\int_{E}fd\mu$ may not be in $X^*$. However,
when the space $X$ is Banach, and function $f:T\rightarrow X^{\ast }$ is w$%
^* $-integrable, $\int_{E}fd\mu \in X^*$ and is called the Gelfand integral
of $f $ over $E$ (see \cite[\S II, Lemma 3.1]{MR0453964} and details therein).

When $X$ is Banach, $L^{\infty }(T,X)$ is the normed space of (equivalence
classes with respect to the relation $f=g$ ae) strongly measurable functions 
$f:T\rightarrow X$, which are essentially bounded; that is, $\Vert x\Vert
_{\infty }:=\text{ess}\sup \{\Vert x(t)\Vert :t\in T\}<\infty .$ A
functional $\lambda ^{\ast }\in L^{\infty }(T,X)^{\ast }$ is called singular
if there exists a sequence of measurable sets $T_{n}$ such that $%
T_{n+1}\subseteq T_{n}$, $\mu (T_{n})\rightarrow 0$ as $n\rightarrow \infty $
and $\lambda ^{\ast }(g\mathds{1}_{T^c_{n}})=0$ for every $g\in L^{\infty
}(T,X)$. We will denote $L^{\text{sing}}(T,X)$ the set of all singular
functionals. It is well-known that each functional $\lambda ^{\ast }\in
L^{\infty }(T,X)^{\ast }$ can be uniquely written as the sum $\lambda ^{\ast
}(\cdot )=\int_{T}\langle \lambda _{1}^{\ast }(t),\cdot \rangle d\mu
(t)+\lambda _{2}^{\ast }(\cdot )$, where $\lambda _{1}^{\ast }\in L_{w^{\ast
}}^{1}(T,X^{\ast })$ and $\lambda _{2}^{\ast }\in L^{\text{sing}}(T,X)$
(see, for example, \cite{MR0467310,0036-0279-30-2-R03}).

A
function $f:T\times X\rightarrow \overline{\mathbb{R}}$ is called a $\tau $%
-normal integrand (or, simply, normal integral when no confusion occurs), if 
$f$ is $\Sigma \otimes \mathcal{B}(X,\tau )$-measurable and the functions $%
f(t,\cdot )$ are lsc for ae $t\in T$. In addition, if  $f(t,\cdot )$ is convex and  proper  for ae $t\in T$, then $f$ is called convex normal integrand.  For
simplicity, we denote $f_{t}:=f(t,\cdot )$. 

Associated with the integrand we consider the integral function ${I}_{f}$ defined
on $X$ as 
\begin{equation*}
	x\in X\rightarrow I_{f}(x):=\int\limits_{T}f(t,x)d\mu (t).
\end{equation*}
A multifunction $G:T\rightrightarrows X$ is called $\Sigma\text{-}\mathcal{B}(X)$-graph measurable (measurable, for simplicity) if its graph, $\grafo G:=\{(t,x)\in
T\times X:x\in G(t)\}$, is an element of $\Sigma \otimes \mathcal{B}(X)$. We
say that $G$ is weakly measurable if for every $x^{\ast }\in X^{\ast }$, $%
t\rightarrow \sigma _{G(t)}(x^{\ast })$ is a measurable function.

The strong and the weak integrals of a (non-necessarily measurable)
multifunction $G:T\rightrightarrows X^{\ast }$ are given respectively by 
\begin{align*}
	\int\limits_{T}G(t)d\mu (t):=& \bigg\{\int\limits_{T}m(t)d\mu (t)\in X^{\ast
	}:m\text{ is integrable}\text{ and }m(t)\in G(t)\;ae\bigg\}, \\
	(w)\text{-}\int\limits_{T}G(t)d\mu (t):=& \bigg\{\int\limits_{T}m(t)d\mu
	(t)\in X^{\ast }:m\text{ is $w^{\ast }$-integrable}\text{ and }m(t)\in
	G(t)\;ae\bigg\}.
\end{align*}
The above definition is called by some authors  the \emph{ Aumman's integral} (see \cite{MR0185073}).

\section{Characterizations under qualification conditions}\label{QUA}

In this section we give the main formulae of the subdifferential of $I_f$.First  let us introduce the following notation. For given $\eta \geq 0$ we denote 
\begin{align*}
\mathcal{I}(\eta ):=\{\ell \in L^{1}(T,\mathbb{R}%
_{+}):\int_T \ell(t)  d\mu(t)\leq \eta \}.
\end{align*}
Furthermore, for $x\in X$ we set $\mathcal{F}(x)$ as the set of all finite-dimensional linear spaces which contains $x$, that is,
\begin{align*}
\mathcal{F}(x):=\{V\subseteq X:V\text{ is a finite-dimensional linear
	space and }x\in V\}.
\end{align*}

\begin{theorem}
	\label{corollaryAsplund}Let\ $X$ be Asplund, and assume that for every
	finite-dimensional subspace $F\subset X,$ the function $f_{|_{F}}:T\times
	F\rightarrow \mathbb{R\cup \{+\infty \}}$ is a convex normal integrand, and
	assume that $I_{f}$ has a continuity point. If $x\in X$ is a point of continuity of  $f_{t}$ for almost every $t$, then for every $\varepsilon \geq 0$ 
	\begin{equation*}
		\partial _{\varepsilon }I_{f}(x)=\bigcup\limits_{\substack{ \epsilon
				_{1}+\epsilon _{2}=\epsilon  \\ \epsilon _{1},\epsilon _{2}\geq 0}}%
		\bigcap_{\gamma >0} \cl \nolimits^{w^{\ast }}\left(
		\bigcup\limits_{\ell \in \mathcal{I}(\epsilon _{1}+\gamma )}\left\{
		(w)\text{-}\int\limits_{T}\partial _{\ell (t)}f_{t}(x)d\mu (t)\right\} \right) +N_{%
			\dom I_{f}}^{\epsilon _{2}}(x).
	\end{equation*}%
	In particular, for $\varepsilon =0$ we have  
	\begin{equation*}
		\partial I_{f}(x)= \cl \nolimits^{w^{\ast }}\left( \left\{
		(w)\text{-}\int\limits_{T}\partial f_{t}(x)d\mu (t)\right\} \right) +N_{{\dom%
			}I_{f}}(x).
	\end{equation*}
\end{theorem}

\begin{proof}
	W.l.o.g.\ we suppose that\ $x=0$ and $\mu (T)<+\infty $. We divide the proof
	into three steps.
	
	\emph{Step 1:} We show in this step that for every $\ell \in \mathcal{I}%
	(\epsilon _{1})$, $\epsilon _{1}\geq 0,$ $t\rightrightarrows \partial _{\ell
		(t)}f_{t}(0)$ is a $w^{\ast }$-measurable multifunction with $w^{\ast }$%
	-compact and convex values. Indeed, the continuity assumption of the $f_{t}$%
	's ensures that the non-empty set $\partial _{\ell (t)}f_{t}(0)$ is $w^{\ast
	}$-compact and convex, as well as $\sigma _{\partial _{\ell
			(t)}f_{t}(0)}(u)=\inf\limits_{\lambda >0}\frac{f(t,0+\lambda u)-f(t,0)+\ell
		(t)}{\lambda }$ for all $u\in X.$ Hence, the function $t\rightarrow \sigma
	_{\partial _{\ell (t)}f(t,0)}(u)$ is measurable, and so is the multifunction 
	$t\rightrightarrows \partial _{\ell (t)}f_{t}(0).$
	
	\emph{Step 2:} We have that for every fixed $L\in \mathcal{F}(0)$ $$\partial _{\epsilon }I_{f}(0)\subseteq {%
		\cl}^{w^{\ast }}\left( \bigcup\limits_{\substack{ \epsilon =\epsilon
			_{1}+\epsilon _{2}  \\ \epsilon _{1},\epsilon _{2}\geq 0  \\ \ell \in 
			\mathcal{I}(\epsilon _{1})}}\left\{ (w)-\displaystyle\int\limits_{T}\partial
	_{\ell (t)}f_{t}(0)d\mu (t)+N_{ \dom I_{f}\cap L}^{\epsilon
		_{2}}(0)\right\} \right). $$  To prove
	this we take $x_{0}\in \inte( \dom I_{f})\cap L$ pick\ $x^{\ast }\in
	\partial _{\epsilon }I_{f}(0)$. By Theorem \cite[Theorem 4.1 and Remark 4.2]{INTECONV} and the continuity
	of $f(t,\cdot )$ there are $\epsilon _{1},\epsilon _{2}\geq 0$ with $%
	\epsilon =\epsilon _{1}+\epsilon _{2}$ and $\ell \in \mathcal{I}(\epsilon
	_{1})$ such that 
	\begin{equation*}
		x^{\ast }\in \int\limits_{T}\big(\partial _{\ell (t)}f_{t}(0)+\spn\{{\dom}I_{f}\cap L\}^{\perp }\big)d\mu (t)+N_{ \dom I_{f}\cap
			L}^{\epsilon _{2}}(0).
	\end{equation*}
	Hence, there exist an integrable function $x_{L}^{\ast }(t)\in \partial
	_{\ell (t)}f(t,0)+\spn\{ \dom I_{f}\cap L\}^{\perp }$ ae, and $%
	\lambda ^{\ast }\in N_{ \dom I_{f}\cap L}^{\epsilon _{2}}(0)$ such
	that $x^{\ast }=\int_{T}x_{L}^{\ast }d\mu +\lambda ^{\ast }$. Now, define
	the multifunction $G:T\rightrightarrows X^{\ast }$ as $$G(t):=\{y^{\ast }\in
	\partial _{\ell (t)}f(t,0):\langle y^{\ast },x_{0}\rangle =\sigma _{\partial
		_{\ell (t)}f(t,0)}(x_{0})\},$$
	where $x_0 $ is a  continuity point of $I_f$.  By \cite[Lemma 4.3]{MR2858100} the multifunction $G$ is $w^{\ast
	}$-measurable (with $w^{\ast }$-compact and convex values), and so
	by  \cite[Corollary 3.11]{MR2858100}  there exists a $w^{\ast }$%
	-measurable selection $x^{\ast }(\cdot )$ of $G$; moreover, we have that $%
	\langle x_{L}^{\ast }(t),x_{0}\rangle \leq \sigma _{\partial _{\ell
			(t)}f(t,0)+\spn\{ \dom I_{f}\cap L\}^{\perp }}(x_{0})=\sigma
	_{\partial _{\ell (t)}f(t,0)}(x_{0})=\langle x^{\ast }(t),x_{0}\rangle $. By
	the continuity of $I_{f}$ we choose $r>0$ such $x_{0}+\mathbb{B}(0,r)\subset %
	\inte( \dom I_{f})$. Then, for every $v\in \mathbb{B}(0,r)$, 
	\begin{align*}
		\langle x^{\ast }(t),v\rangle \leq & f(t,x_{0}+v)-f(t,0)+\ell (t)-\langle
		x^{\ast }(t),x_{0}\rangle \\
		=& f(t,x_{0}+v)-f(t,0)+\ell (t)-\sigma _{\partial _{\ell (t)}f(t,0)}(x_{0})
		\\
		\leq & f(t,x_{0}+v)-f(t,0)+\ell (t)-\langle x_{L}^{\ast }(t),x_{0}\rangle ,
	\end{align*}%
	and so, 
	\begin{align*}
		\int\limits_{T}|\langle x^{\ast }(t),v\rangle |d\mu (t)&\leq
		\int\limits_{T}\left( \max \{f(t,x_{0}+v),f(t,x_{0}-v)\}-f(t,0)\right)d\mu(t)\\& +	\int\limits_{T}\left(  \ell
		(t)-\langle x_{L}^{\ast }(t),x_{0}\rangle \right) d\mu (t)\\&<+\infty ;
	\end{align*}%
	that is, $x^{\ast }(\cdot )$ is Gelfand integrable ($X$ is Banach). This
	last inequality\ implies that 
	\begin{equation*}
		C:=\bigcup\limits_{\substack{ \epsilon =\epsilon _{1}+\epsilon _{2}  \\ %
				\epsilon _{1},\epsilon _{2}\geq 0  \\ \ell \in \mathcal{I}(\epsilon _{1})}}%
		\left\{ (w)\text{-}\int\limits_{T}\partial _{\ell (t)}f_{t}(0)d\mu (t)+N_{{	\dom}I_{f}\cap L}^{\epsilon _{2}}(0)\right\} \neq \emptyset .
	\end{equation*}%
	Because $\langle x^{\ast },x_{0}\rangle =\langle \int_{T}x_{L}^{\ast
	}(t)d\mu (t),x_{0}\rangle +\langle \lambda ^{\ast },x_{0}\rangle \leq
	\langle \int_{T}x^{\ast }(t)d\mu (t)+\lambda ^{\ast },x_{0}\rangle ,$ from
	the arbitrariness of $x^{\ast }\in \partial _{\epsilon }I_{f}(0)$ and $%
	x_{0}\ $in $\inte( \dom I_{f})\cap L$ we get\ 
	\begin{equation*}
		\sigma _{\partial _{\epsilon }I_{f}(0)}(x_{0})\leq \sigma _{C}(x_{0})\text{
			for every }x_{0}\in \inte( \dom I_{f})\cap L,
	\end{equation*}%
	which also implies by\ usual arguments that $\sigma _{\partial _{\epsilon
		}I_{f}(0)}(u)\leq \sigma _{C}(u)$ for all $u\in X,$ and the desired relation
	holds.
	
	\emph{Step 3: }We complete the proof of the theorem. We show now that\ $$%
	\partial _{\epsilon }I_{f}(0)=\bigcup\limits_{\substack{ \epsilon =\epsilon
			_{1}+\epsilon _{2}  \\ \epsilon _{1},\epsilon _{2}\geq 0}}\cl^{w^{\ast
	}}\bigcup\limits_{\substack{ \gamma \in \lbrack 0,\epsilon -\epsilon _{1}] 
			\\ \ell \in \mathcal{I}(\epsilon _{1}+\gamma )}}\left\{
	(w)\text{-}\int\limits_{T}\partial _{\ell (t)}f_{t}(0)d\mu (t)\right\} +N_{{%
			\dom}I_{f}}^{\epsilon _{2}}(0).$$ We take $x^{\ast }\in \partial _{\epsilon
	}I_{f}(0)$, so that by Step 2 there are nets of numbers $\epsilon
	_{1,L,V},\epsilon _{2,L,V}\geq 0$ with $\epsilon _{1,L,V}+\epsilon
	_{2,L,V}=\epsilon $ and $\ell _{L,V}\in \mathcal{I}(\epsilon _{1,L,V})$, together with
	vectors $$x_{L,V}^{\ast }\in (w)\text{-}\int_{T}\partial _{\ell
		_{0L,V}(t)}f_{t}(0)d\mu (t) \text{ and }\lambda _{L,V}^{\ast }\in N_{ \dom %
		I_{f}\cap L}^{\epsilon _{2,L,V}}(0),$$ indexed by $(L,V)\in \mathcal{F}%
	(0)\times \mathcal{N}_{0}(w^{\ast })$ such that $x^{\ast }=\lim
	x_{L,V}^{\ast }+\lambda _{L,V}^{\ast }$, where the limits is taken with respect to the $w^\ast$-topology. We may assume that $\epsilon
	_{1,L,V}\rightarrow \epsilon _{1}$ and $\epsilon _{2,L,V}\rightarrow
	\epsilon _{2},$ with $\epsilon _{1}+\epsilon _{2}=\epsilon $. Now, by the
	continuity of $I_{f}$ at $x_{0}$ there is some $r>0$ such that for every $%
	v\in B(0,r)$ and for every $L\ni x_{0}$ (w.l.o.g.)%
	\begin{align*}
		\langle x_{L,V}^{\ast },v\rangle \leq & I_{f}(x_{0}+v)-I_{f}(0)+\epsilon
		_{1,L,V}-\langle x_{L,V}^{\ast },x_{0}\rangle \\
		\leq & I_{f}(x_{0}+v)-I_{f}(0)+\epsilon _{1,L,V}-\langle x^{\ast
		},x_{0}\rangle +\langle \lambda _{L,V}^{\ast },x_{0}\rangle +1 \\
		\leq & I_{f}(x_{0}+v)-I_{f}(0)+\epsilon -\langle x^{\ast },x_{0}\rangle +1.
	\end{align*}%
	Therefore, we may suppose that $(x_{L,V}^{\ast })$ $w^{\ast }$-converges to
	some $y^{\ast }\in X^{\ast }$ and that $(\lambda _{L,V}^{\ast })$ $w^{\ast }$%
	-converges to some $\nu ^{\ast }\in X^{\ast };$ hence, $\nu ^{\ast }\in N_{%
		\dom I_{f}}^{\epsilon _{2}}(0).$ Finally, if $\epsilon =0,$ then the
	conclusion follows. Otherwise, if $\epsilon >0$, for every\ $\gamma >0$ we
	obtain that $\ell _{L,V}\in \mathcal{I}(\epsilon _{1}+\gamma )$ for a
	co-final family of indices $L,V$, and so the desired inclusion follows.
\end{proof}

The next result is a generalization of \cite{MR0372610}, see also \cite%
{MR0322531,MR0217502,Levin1968,zbMATH03263581} for other versions. First, we
need to make a remark about  the relation between the continuity of $I_{f}$ and
the $f_{t}$'s.

\begin{remark}
	\label{remc}\emph{Assume that either }$X$\emph{\ is Suslin or }$(T,\Sigma )=(%
	\mathbb{N},\mathcal{P}(\mathbb{N})).$\emph{\ Due to the relation}%
	\begin{equation*}
		{\inte}(\overline{{\dom}I_{f}})\subset {\inte}(%
		\overline{{\dom}f_{t}})\text{ for ae }t\in T,
	\end{equation*}%
	\emph{the continuity hypothesis used in  \cref{Exactintegralformula} below
		is equivalent to the continuity of the functions }$I_{f}$\emph{\ and }$%
	f(t,\cdot ),$\emph{\ }$t\in T,$\emph{\ at some common point. In particular,
		in the finite-dimensional setting, the continuity of }$I_{f}$\emph{\ alone
		ensures the continuity of the }$f_{t}$\emph{'s on the interior of their
		domains. }
\end{remark}

\begin{theorem}
	\label{Exactintegralformula}
	Assume that either $X$ is a Suslin space, or $%
	(T,\Sigma )=(\mathbb{N},\mathcal{P}(\mathbb{N})).$
	If each one of the functions $I_{f}$ and $f(t,\cdot ),$ $t\in T,$ is\
	continuous at some point. Then for every $x\in X$ and $\epsilon \geq 0.$\ 
	\begin{equation*}
		\partial _{\epsilon }I_{f}(x)=\bigcup\limits_{\substack{ \epsilon =\epsilon
				_{1}+\epsilon _{2}  \\ \epsilon _{1},\epsilon _{2}\geq 0  \\ \ell \in 
				\mathcal{I}(\epsilon _{1})}}\left\{ (w)\text{-}\int\limits_{T}\partial _{\ell
			(t)}f_{t}(x)d\mu (t)+N_{{\dom}I_{f}}^{\epsilon _{2}}(x)\right\} .
	\end{equation*}
\end{theorem}
\begin{proof}
	We fix $x\in X$ and $\varepsilon \geq 0,$ and choose a common continuity
	point $x_{0}$ of\ $I_{f}$ and the $f_{t}$'s (see  \cref{remc}).\ The
	right-hand side is straightforwardly included in $\partial _{\epsilon
	}I_{f}(x)$, and so we focus on the opposite inclusion. W.l.o.g. we may
	assume that $x=0$, $\partial _{\epsilon }\neq \emptyset ,$ $I_{f}(0)=0,$ as
	well as $\mu (T)=1$. Take\ $x^{\ast }\in \partial _{\epsilon }I_{f}(0)$, and
	fix\ a sequence of positive functions $(\eta _{n})_{n}\subset L^{\infty }(T,%
	\mathbb{R})$ which converges to zero. By \cite[Theorem 5.1]{INTECONV}, there
	exists a net of integrable functions $w_{n,L,V}^{\ast }(t)\in \partial
	_{\ell _{n,L,V}(t)+\eta _{n}(t)}f_{t}(0)+N_{ \dom I_{f}\cap
		L}^{\epsilon _{n,L,V,2}}(0)$, with $n\in \mathbb{N}$, $L\in \mathcal{F}(0)$, 
	$V\in \mathcal{N}_{0}$ and $\ell _{n,L,V}\in \mathcal{I}(\epsilon
	_{n,L,V,1}) $ such that 
	\begin{equation}
		x^{\ast }=\lim_{n,L,V}\int_{T}w_{n,L,V}^{\ast }(t)d\mu (t)\text{ and }%
		\epsilon _{n,L,V,1}+\epsilon _{n,L,V,2}=\epsilon .  \label{tt}
	\end{equation}%
	Next, as in the proof of \cite[Theorem 5.1]{INTECONV}, we find measurable
	functions\ $x_{n,L,V}^{\ast }$ and $\lambda _{n,L,V}^{\ast }$ such that $%
	x_{n,L,V}^{\ast }(t)\in \partial _{\eta _{n}(t)+\ell _{n,L,V}(t)}f(t,0)$ and 
	$\lambda _{n,L,V}^{\ast }(t)\in N_{ \dom I_{f}\cap L}^{\epsilon
		_{n,L,V,2}}(0)$ for ae, with $w_{n,L,V}^{\ast }(t)=x_{n,L,V}^{\ast
	}(t)+\lambda _{n,L,V}^{\ast }(t)$. To simplify the notation, we just write $%
	w_{n,i}^{\ast }(\cdot )$, $x_{n,i}^{\ast }(\cdot )$, $\epsilon _{n,i,1}$, $%
	\epsilon _{n,i,2}$, $\lambda _{n,i}^{\ast }(\cdot )$ and $\epsilon
	_{n,i}:=\ell _{n,i}+\eta _{n}$, $i\in I:=\mathcal{F}(0)\times \mathcal{N}%
	_{0},$ where $\mathbb{N\times }I$ is endowed with the partial order "$%
	\preceq $" given by $(n_{1},L_{1},V_{1})\preceq (n_{2},L_{2},V_{2})$ iff $%
	n_{1}\leq n_{2},$ $L_{1}\subset L_{2}$ and $V_{1}\supset V_{2}.$
	
	The rest of the proof is divided into three steps.

	\emph{Step 1:} W.l.o.g. on $n,i,$ there exists $U\in \mathcal{N}_{0}$ such
	that
	\begin{equation}
		\begin{array}{rl}
			m&:=\sup\limits_{v\in U,n\in \mathbb{N},i\in I}\displaystyle\int\limits_{T}\langle
			x_{n,i}^{\ast }(t),v\rangle d\mu (t)<+\infty \\
			
			m_{x}&:=\sup\limits_{n,i}%
			\displaystyle\int\limits_{T}|\langle x_{n,i}^{\ast }(t),x\rangle |d\mu (t)<+\infty
			\;\;\forall x\in X. 
		\end{array}
		\label{cl1}
	\end{equation}%
	Indeed, we choose\ $U\in \mathcal{N}_{0}$ such that $\sup\limits_{v\in
		V}I_{f}(x_{0}+v)<+\infty $. Then for every $n\in \mathbb{N}$, $i\in I$ and $%
	v\in U$ 
	\begin{align}
		\langle x_{n,i}^{\ast }(t),v\rangle & \leq f(t,x_{0}+v)-f(t,0)-\langle
		x_{n,i}^{\ast }(t),x_{0}\rangle +\epsilon _{n,i}(t)  \notag \\
		& =f(t,x_{0}+v)-f(t,0)-\langle w_{n,i}^{\ast }(t)-\lambda _{n,i}^{\ast
		}(t),x_{0}\rangle +\epsilon _{n,i}(t)  \notag \\
		& \leq f(t,x_{0}+v)-f(t,0)-\langle w_{n,i}^{\ast }(t),x_{0}\rangle +\epsilon
		_{n,i,2}+\epsilon _{n,i}(t).  \label{tt1}
	\end{align}%
	But, by \cref{tt} and the definition of $\epsilon _{n,i}$ $(\epsilon
	_{n,i}=\ell _{n,i}+\eta _{n})$, we may suppose that for all $n$ and $i$, 
	\begin{align}
		\begin{array}{rl}
			-\int_{T}\langle w_{n,i}^{\ast }(t),x_{0}\rangle +\int_{T}\epsilon
			_{n,i}(t)d\mu (t)+\epsilon _{n,i,2}&\leq -\langle x^{\ast },x_{0}\rangle
			+\epsilon +\int_{T}\eta _{n}d\mu +\frac{1}{2}\\
			&\leq -\langle x^{\ast
			},x_{0}\rangle +\epsilon +1, 
		\end{array}
		\label{tt2}
	\end{align}%
	and so \cref{tt1} leads to $\sup\limits_{v\in U,n\in \mathbb{N},i\in
		I}\int_{T}\langle x_{n,i}^{\ast }(t),v\rangle d\mu (t)<+\infty ,$ which is
	the first\ part of \cref{cl1}. Now, we\ define the sets $%
	T_{n,i,v}^{+}:=\{t\in T:\langle x_{n,i}^{\ast }(t),v\rangle \geq 0\}$, $%
	T_{n,i,v}^{-}:=\{t\in T:\langle x_{n,i}^{\ast }(t),v\rangle <0\}$, $n\in 
	\mathbb{N}$, $i\in I$ and $v\in U.$ Then, using \cref{tt1},%
	\begin{align*}
		\int_{T}|\langle x_{n,i}^{\ast }(t),v\rangle |d\mu (t)=&
		\int_{T_{n,i,v}^{+}}\langle x_{n,i}^{\ast }(t),v\rangle d\mu
		(t)-\int_{T_{n,i,v}^{-}}\langle x_{n,i}^{\ast }(t),v\rangle d\mu (t) \\
		\leq & \int_{T_{n,i,v}^{+}} ( f(t,x_{0}+v)-f(t,0)-\langle w_{n,i}^{\ast
		}(t),x_{0}\rangle ) d\mu (t) \\& +\int_{T_{n,i,v}^{+}}( \epsilon _{n,i,2}+\epsilon _{n,i})  d\mu (t) \\
		& +\int_{T_{n,i,v}^{-}} ( f(t,x_{0}-v)-f(t,0)-\langle w_{n,i}^{\ast
		}(t),x_{0}\rangle) d\mu (t) \\&  + \int_{T_{n,i,v}^{-}}(\epsilon _{n,i,2}+\epsilon _{n,i}) d\mu (t) \\
		=& \int_{T_{n,i,v}^{+}}f(t,x_{0}+v)d\mu
		(t)+\int_{T_{n,i,v}^{-}}f(t,x_{0}-v)d\mu (t) \\
		& -\langle x^{\ast },x_{0}\rangle +\epsilon +1\text{ \ \ (by \cref{tt2})} \\
		\leq & \int_{T}|f(t,x_{0}+v)|d\mu (t)\\&+\int_{T}|f(t,x_{0}-v)|d\mu (t)-\langle
		x^{\ast },x_{0}\rangle +\epsilon +1<+\infty ,
	\end{align*}%
	and the second part in \cref{cl1} follows since\ $U$ is absorbent.
	
	\emph{Step 2:} There exist $\epsilon _{1},\epsilon _{2},\epsilon _{3}\geq 0$
	with $\epsilon _{1}+\epsilon _{2}+\epsilon _{3}\leq \varepsilon ,$
	neighborhood $U\in \mathcal{N}_{0}$, $\lambda _{1}^{\ast }\in N_{{\dom%
		}I_{f}}^{\epsilon _{3}}(0),$ linear functions $F_{1}:X\rightarrow L^{1}(T,%
	\mathbb{R})$ and $F_{2}:X\rightarrow L^{\text{sing}}(T,\mathbb{R})$,
	together with\ elements $\ell \in L^{1}(T,\mathbb{R}_{+})$ and $s\in L^{%
		\text{sing}}(T,\mathbb{R})$ such that (w.l.o.g. on $n$ and $i$):
	
	\begin{enumerate}[label={(\roman*)},ref={(\roman*)}] 
		
		\item \label{e0}$\lambda _{1}^{\ast }=\lim_{n,i}\int_{T}\lambda _{n,i}^{\ast
		}(t)d\mu (t).$
		
		\item \label{e1}For every $x\in X,$\ $(\langle x_{n,i}^{\ast }(\cdot
		),x\rangle )_{n,i}\subset L^{1}(T,\mathbb{R})\subset L^{\infty }(T,\mathbb{R}%
		)^{\ast }$ and $\langle x_{n,i}^{\ast }(\cdot ),x\rangle \rightarrow
		F_{1}(x)+F_{2}(x)$ wrt to the $w^{\ast }$-topology in $L^{\infty }(T,\mathbb{%
			R})^{\ast }$.
		
		\item \label{e2}$(\epsilon _{n,i}(\cdot ))\subset L^{1}(T,\mathbb{R})\subset
		L^{\infty }(T,\mathbb{R})^{\ast }$ and $\epsilon _{n,i}(\cdot
		)\longrightarrow \ell +s$ wrt to the $w^{\ast }$-topology in $L^{\infty }(T,%
		\mathbb{R})^{\ast }$.
		
		\item \label{e3}$\sup\limits_{v\in U}\int\limits_{T}F_{1}(v)d\mu (t)<+\infty
		,$ $\sup\limits_{v\in U}F_{2}(v)(\mathds{1}_{T})<+\infty $.
		
		\item \label{e4}$\varepsilon _{1}:=\int_{T}\ell (t)d\mu (t),$ $\varepsilon
		_{2}:=s(\mathds{1}_{T})\geq 0$.
		
		\item \label{e5}For every $(x,A)\in  X\times \Sigma$
		\begin{equation}
			\int_{A}F_{1}(x)d\mu (t)\leq \int_{A}f(t,x)d\mu (t)-\int_{A}f(t,0)d\mu
			(t)+\int_{A}\ell (t)d\mu (t),
			\label{equationsubdiferential}
		\end{equation}%
		and 
		\begin{equation}
			F_{2}(x)(\mathds{1}_{T})\leq s(\mathds{1}_{T}),\text{ \ for all }x\in 
			\dom I_{f}.  \label{equationsubdiferential2}
		\end{equation}
	\end{enumerate}
	
	Consider $U$, $m$ and $(m_{x})_{x\in X}$ as in the previous step, and denote
	by $B$ the unit ball in the dual space of $L^{\infty }(T,\mathbb{R}).$ From \cref{cl1} and the definition of $x^{\ast },$ we obtain the existence of $%
	\lambda _{1}^{\ast }\in X^{\ast }$ such that (w.l.o.g.) $\lambda _{1}^{\ast
	}=\lim_{n,i}\int_{T}\left\langle \lambda _{n,i}^{\ast }(t),\cdot
	\right\rangle d\mu (t).$ Moreover, given an\ $x\in  \dom I_{f}$ we
	write, since $\lambda _{n,i}^{\ast }(t)\in N_{ \dom I_{f}\cap
		L}^{\epsilon _{n,i,2}}(0)$ and\ $x\in L$ (for $L$ large enough), 
	\begin{equation*}
		\left\langle \lambda _{1}^{\ast },x\right\rangle
		=\lim_{n,i}\int_{T}\left\langle \lambda _{n,i}^{\ast }(t),x\right\rangle
		d\mu (t)\leq \lim_{n,i}\epsilon _{n,i,2}=:\varepsilon _{3},
	\end{equation*}%
	and so $\lambda _{1}^{\ast }\in N_{ \dom I_{f}}^{\epsilon _{3}}(0);$
	hence, \cref{e0} follows.
	
	Next, by Thychonoff's theorem the space $$\mathfrak{X}:=\prod\limits_{x\in
		X}(m_{x}B,w^{\ast }((L^{\infty }(T,\mathbb{R}))^{\ast },L^{\infty }(T,%
	\mathbb{R})))$$ is a compact space with respect to the product topology, and so
	w.l.o.g. we may assume that the net $(\langle x_{n,i}^{\ast }(\cdot
	),x\rangle )_{x\in X}\in \mathfrak{X}$, $(n,i)\in \mathbb{N\times }I,$
	converges to some $(F(x))_{x\in X},$ where $F:X\rightarrow (L^{\infty }(T,%
	\mathbb{R}))^{\ast }$ is a linear function. Using the classical decomposition $%
	(L^{\infty }(T,\mathbb{R}))^{\ast }=L^{1}(T,\mathbb{R})\oplus L^{\text{sing}%
	}(T,\mathbb{R}),$ for every $x\in X$ we write $F(x)=F_{1}(x)+F_{2}(x),$
	where $F_{1}:X\rightarrow L^{1}(T,\mathbb{R})$ and $F_{2}:X\rightarrow L^{%
		\text{sing}}(T,\mathbb{R})$ are two linear functions, and \cref{e1} follows.
	Similary, since $(\epsilon _{n,i}(\cdot ))$ is bounded in $L^{1}(T,\mathbb{R}%
	)$ we may assume that it converges to some\ $l+s$, with $l\in L^{1}(T,%
	\mathbb{R})$ and $s\in L^{\text{sing}}(T,\mathbb{R}),$ such that for all $%
	G\in \Sigma $ 
	\begin{align}
		\int_{G}\ell (t)d\mu (t)+s(\mathds{1}_{G})&=\lim_{n,i}\int_{G}\epsilon
		_{n,i}(t)d\mu (t)=\lim_{n,i}\int_{G}(\ell _{n,i}(t)+\eta _{n}(t))d\mu
		(t)\\&=\lim_{n,i}\int_{G}\ell _{n,i}(t)d\mu (t)\leq \varepsilon -\varepsilon
		_{3}\mathnormal{.}  \label{we0}
	\end{align}%
	Fix\ $x\in X.$ Since $F_{2}(x),$ $s\in L^{\text{sing}}(T,\mathbb{R}),$ there
	exists a sequence of measurable sets $T_{n}(x)$ such that $\mu (T\backslash
	\bigcup T_{n}(x))=0$ and 
	\begin{equation*}
		F_{2}(x)(g\mathds{1}_{T_{n}(x)})=0,\text{ }s(g\mathds{1}_{T_{n}(x)})=0\text{
			for all }n\in \mathbb{N}\text{ and }g\in L^{\infty }(T,\mathbb{R}).
	\end{equation*}%
	Thus, by replacing in \cref{we0} the set $G$ by $T_{k}(x),$ $k\geq 1,$ and $%
	T\setminus \cup _{1\leq k\leq n}T_{k}(x),$ respectively, and taking the
	limit on $k$, \cref{e2} and \cref{e4} follow.
	
	Now, for every\ $v\in U,$ $G\in \Sigma $, $n\in \mathbb{N}$ and $(n,i)\in 
	\mathbb{N\times }I$ (recall \cref{cl1}) 
	\begin{equation*}
		\int\limits_{G}\langle x_{n,i}^{\ast }(t),x\rangle d\mu (t)\leq
		\int\limits_{G}f(t,x)d\mu (t)-\int\limits_{G}f(t,0)d\mu
		(t)+\int\limits_{G}\epsilon _{n,i}(t)d\mu (t),
	\end{equation*}%
	\begin{equation*}
		\int\limits_{G}\langle x_{n,i}^{\ast }(t),v\rangle d\mu (t)\leq m.
	\end{equation*}%
	So, by taking the limit we get 
	\begin{equation}
		\int\limits_{G}F_{1}(x)d\mu (t)+F_{2}(x)(\mathds{1}_{G})\leq
		\int\limits_{G}f(t,x)d\mu (t)-\int\limits_{G}f(t,0)d\mu (t)+\int_{G}l(t)d\mu
		(t)+s(\mathds{1}_{G}),  \label{we}
	\end{equation}%
	\begin{equation}
		\int\limits_{G}F_{1}(v)d\mu (t)+F_{2}(v)(\mathds{1}_{G})\leq m.  \label{we2}
	\end{equation}
	In particular, for\ $A\in \Sigma $ and\ $G_{n}=A\cap T_{n}(x)$ we get\ 
	\begin{equation*}
		\int\limits_{G_{n}}F_{1}(x)d\mu (t)\leq \int\limits_{G_{n}}f(t,x)d\mu
		(t)-\int\limits_{G_{n}}f(t,0)d\mu (t)+\int\limits_{G_{n}}l(t)d\mu (t),
	\end{equation*}%
	\begin{equation*}
		\int\limits_{G_{n}}F_{1}(v)d\mu (t)=\int\limits_{G_{n}}F_{1}(v)d\mu
		(t)+F_{2}(v)(\mathds{1}_{G_{n}})\leq m
	\end{equation*}%
	which as $n\rightarrow \infty $ gives us 
	\begin{equation}
		\int\limits_{A}F_{1}(v)d\mu (t)\leq m,  \label{q}
	\end{equation}%
	and\ 
	\begin{equation*}
		\int\limits_{A}F_{1}(x)d\mu (t)\leq \int\limits_{A}f(t,x)d\mu
		(t)-\int\limits_{A}f(t,0)d\mu (t)+\int\limits_{A}l(t)d\mu (t),
	\end{equation*}%
	yielding the first part in \cref{equationsubdiferential}. Now, for $%
	G_{n}=T\backslash \bigcup\limits_{i=1}^{n}T_{i}(x)$ we have that $$F_{2}(x)(%
	\mathds{1}_{G_{n}})=F_{2}(x)(\mathds{1}_{T})\text{ and } s(\mathds{1}_{G_{n}})=s(%
	\mathds{1}_{T}).$$  Furthermore, for $v\in U$ and $x\in  \dom I_{f}$
	\begin{equation*}
		\int\limits_{G_{n}}F_{1}(v)d\mu (t),\int\limits_{G_{n}}F_{1}(x)d\mu
		(t),\int\limits_{G_{n}}f(t,x)d\mu (t),\int\limits_{G_{n}}f(t,0)d\mu
		(t),\int\limits_{G_{n}}l(t)d\mu (t)\rightarrow _{n}0,
	\end{equation*}%
	and so, \cref{we} and \cref{we2} yield $F_{2}(x)(\mathds{1}_{T})\leq s(\mathds{1}_{T})$, and 
	\begin{equation}
		F_{2}(v)(\mathds{1}_{T})\leq m.  \label{q2}
	\end{equation}%
	We get \cref{e5}, while \cref{e3} follows from \cref{q} and \cref{q2}.
	
	\emph{Step 3: }Let $\ell ,$ $s$, $\epsilon _{1},\epsilon _{2},$ $\epsilon
	_{3},$ $U\in \mathcal{N}_{0}$, $\lambda _{1}^{\ast }\in N_{ \dom %
		I_{f}}^{\epsilon _{3}}(0),$ $F_{1},$ and $F_{2}\ $be as in step 2. We show
	the existence of a weakly integrable function $y^{\ast }:T\rightarrow
	X^{\ast }$ such that $y^{\ast }(t)\in \partial _{\ell (t)}f_{t}(0)$ ae, and $%
	x^{\ast }-\int_{T}y^{\ast }d\mu \in N_{ \dom I_{f}}^{\epsilon
		_{2}}(0) $.
	Assume first that $(T,\Sigma )=(\mathbb{N},\mathcal{P}(\mathbb{N})).$ In
	this case, on the one hand we\ take $y^{\ast }(t):=F_{1}(\cdot )(t),$ $t\in
	T.$ Then, for every $x\in X,$ by \cref{equationsubdiferential} we have that
	for all $t\in T$ 
	\begin{equation*}
		\left\langle y^{\ast }(t),x\right\rangle =F_{1}(x)(t)=(\mu
		(t))^{-1}\int_{\{t\}}F_{1}(x)d\mu (t)\leq f(t,x)-f(t,0)+\ell (t),
	\end{equation*}%
	which, by taking into account\ the continuity assumption on $f(t,\cdot )$,
	shows that $y^{\ast }(t)\in \partial _{\ell (t)}f_{t}(0).$ Also, since  $%
	\int_{T}\left\vert \left\langle y^{\ast },x\right\rangle \right\vert d\mu
	=\int_{T}\left\vert F_{1}(x)(t)\right\vert d\mu (t)<+\infty ,$ for all $x\in
	X,$ and (by \cref{e3}) 
	\begin{equation*}
		\sup\limits_{v\in U}\int_{T}\left\langle y^{\ast },v\right\rangle d\mu
		=\sup\limits_{v\in U}\int\limits_{T}F_{1}(v)d\mu (t)<+\infty ,
	\end{equation*}%
	it follows that $y^{\ast }:=\int_{T}y^{\ast }d\mu \in (w)$-$%
	\int\limits_{T}\partial _{\ell (t)}f_{t}(x)d\mu (t).$ On the other hand,\ we
	take $\lambda ^{\ast }:=\lambda _{1}^{\ast }+\lambda _{2}^{\ast },$ with $%
	\lambda _{2}^{\ast }:=F_{2}(\cdot )(\mathds{1}_{T})$ ($\in X^{\ast },$ by %
	\cref{e3}), so that for all $x\in  \dom I_{f}$ (using \cref{equationsubdiferential2})%
	\begin{equation*}
		\left\langle \lambda _{2}^{\ast },x\right\rangle =F_{2}(x)(\mathds{1}%
		_{T})\leq s(\mathds{1}_{T})=\varepsilon _{2}.
	\end{equation*}%
	Hence, $\lambda ^{\ast }\in N_{ \dom I_{f}}^{\epsilon _{3}}(0)+N_{%
		\dom I_{f}}^{\epsilon _{2}}(0)\subset N_{ \dom %
		I_{f}}^{\epsilon _{2}+\epsilon _{3}}(0),$ and so we get 
	\begin{eqnarray*}
		x^{\ast } &=&\lim_{n,L,V}\int_{T}(x_{n,L,V}^{\ast }(t)+\lambda
		_{n,L,V}^{\ast }(t))d\mu (t) \\
		&=&\lim_{n,L,V}\int_{T}x_{n,L,V}^{\ast }(t)d\mu (t)+\lim_{n,L,V}\lambda
		_{n,L,V}^{\ast }(t)d\mu (t) \\
		&=&y^{\ast }+\lambda _{2}^{\ast }+\lambda _{1}^{\ast }\in
		(w)\text{-}\int\limits_{T}\partial _{\ell (t)}f_{t}(x)d\mu (t)+N_{ \dom %
			I_{f}}^{\epsilon _{2}+\epsilon _{3}}(0),
	\end{eqnarray*}%
	which ensures the desired inclusion.
	
	We treat now\ the case when $X,X^{\ast }$ are Suslin spaces. We choose a
	countable set $D$ such that $X=\{\lim x_{n}:(x_{n})_{n\in \mathbb{N}}\subset
	D\}.$ Equivalently, we can take an at most countable\ family of linearly
	independent vectors\ $\{e_{i}\}_{i=1}^{\infty }$ such that $L:=\spn%
	\{e_{i}\}_{i=1}^{\infty }\supseteq D$ and $L_{n}:=\spn\{e_{i}\}_{i=1}^{n}\ni
	x_{0}$ for all $n\geq 1$ (recall that $x_{0}$ is a common continuity point
	of\ $I_{f}$ and the $f_{t}$'s). As in the previous discrete case, we take\ $%
	\lambda _{2}^{\ast }=F_{2}(\cdot )(\mathds{1}_{T})$. So, by argyuing as
	above we obtain that $\lambda ^{\ast }:=\lambda _{1}^{\ast }+\lambda
	_{2}^{\ast }\in N_{ \dom I_{f}}^{\epsilon _{2}+\epsilon _{3}}(0).$
	Next, we consider a sequence of functions $(b_{n})_{n}$ such that each $%
	b_{n} $ is in the class of equivalence $F_{1}(e_{n}),$ and define\ for every 
	$t\in T$ a linear function $y_{t}^{\ast }:L\rightarrow \mathbb{R}$ as $%
	\langle y_{t}^{\ast },z\rangle =\sum\limits_{i=1}^{n}\alpha _{i}b_{i}(t)$ $%
	(\in F_{1}(\sum\limits_{i=1}^{n}\alpha _{i}e_{n})(t)=F_{1}(z)(t)),$ where $%
	z=\sum\limits_{i=1}^{n}\alpha _{i}e_{i}$, $\alpha _{i}\in \mathbb{R}$. We
	notice that for every $z\in L,$ $t\rightarrow \langle y_{t}^{\ast },z\rangle 
	$ is measurable. Now, given\ $z\in L_{\mathbb{Q}}:=\bigoplus\limits_{i=1}^{%
		\infty }\mathbb{Q}e_{i},$ we define $T_{z}:=\{t\in T\mid \langle y_{t}^{\ast
	},z\rangle \leq f(t,z)-f(t,0)+\ell (t)\}$ and $\tilde{T}:=\bigcap\limits_{z%
		\in L_{\mathbb{Q}}}T_{z};$ hence, $\mu (T\backslash \tilde{T})=0$, by
	\cref{equationsubdiferential}. Now, because $\inte \dom %
	f_{t}\cap L_{\mathbb{Q}}\neq \emptyset $ and $f_{t}$ is continuous on $\inte%
	\dom f_{t},$ it follows that 
	\begin{equation*}
		\langle y_{t}^{\ast },z\rangle \leq f(t,z)-f(t,0)+l(t)\text{ \ for all }t\in 
		\tilde{T}\ \text{and }z\in L;
	\end{equation*}%
	in particular, $y_{t}^{\ast }$ is a continuous linear functional on $L$ for
	every $t\in \tilde{T}$. Now, by the Hahn-Banach theorem, we can extend $%
	y_{t}^{\ast }$ to a continuous linear functional on $X$, denoted by $y^{\ast
	}(t)$, such that 
	\begin{equation*}
		\langle y^{\ast }(t),z\rangle \leq f(t,z)-f(t,0)+l(t)\text{ \ for all }t\in 
		\tilde{T}\ \text{and }z\in L.
	\end{equation*}%
	We notice that $y^{\ast }(\cdot )$ is weakly measurable, because for every\ $%
	x\in X$, since $D\subset L$ there exists a sequence of element $x_{n}\in L$
	such that $x_{n}\rightarrow x$ and, hence, $\langle y^{\ast }(t),x\rangle
	=\lim_{n}\langle w_{t}^{\ast },x_{n}\rangle $ is measurable as is each
	function $t\rightarrow \langle w_{t}^{\ast },x_{n}\rangle $. Moreover, by
	the continuity of $f_{t}$ on $\inte \dom f_{t},$ the last inequality
	above holds on $X,$ and this gives us $y^{\ast }(t)\in \partial
	_{l(t)}f_{t}(0)$ for all $t\in \tilde{T}.$ Now, by the continuity of $I_{f}$
	and using similar arguments as those in the proof of \cref{corollaryAsplund}, it is not difficult to show that $%
	y^{\ast }(t)$ is weakly integrable and that $y^{\ast }:=\int_{T}y^{\ast
	}(t)\mu (t)$ defines a continuous linear operator on $X$, showing that $%
	y^{\ast }\in (w)$-$\int\limits_{T}\partial _{\ell (t)}f_{t}(x)d\mu (t).$
	Whence $x^{\ast }=y^{\ast }+\lambda ^{\ast }\in (w)$-$\int\limits_{T}%
	\partial _{\ell (t)}f_{t}(x)d\mu (t)+N_{ \dom I_{f}}^{\epsilon
		_{2}+\epsilon _{3}}(0)$. The proof of the theorem is finished.
\end{proof}

Now we give a formula for the $\epsilon$-subdifferential of $I_f$ under a stronger qualification condition. Consider a Banach space $X$, we denote by $\hat{I}_{f}: L^{\infty
}(T,X) \to \Rex$ the integral functional
\begin{equation*}
	x(\cdot )\in L^{\infty
	}(T,X) \rightarrow \hat{I}_{f}(x(\cdot
	)):=\int\limits_{T}f(t,x(t))d\mu (t).
\end{equation*}%

\begin{theorem}
	\label{Banachintegralformula}Assume that $X$ is a separable Banach space.
	If\ $\hat{I}_{f}$ is bounded above on some neighborhood wrt $(L^{\infty
	}(T,X),\Vert \cdot \Vert )$ of some constant function $x_{0}(\cdot )\equiv
	x_{0}\in X,$ then for all $x\in X$ and $\varepsilon \geq 0$ we have that 
	\begin{equation*}
		\partial _{\epsilon }I_{f}(x)=\bigcup\limits_{\substack{ \epsilon =\epsilon
				_{1}+\epsilon _{2}  \\ \epsilon _{1},\epsilon _{2}\geq 0  \\ \ell \in 
				\mathcal{I}(\epsilon _{1})}}\left\{ \int\limits_{T}x^{\ast }d\mu :x^{\ast
		}\in L_{w^{\ast }}^{1}(T,X^{\ast }),\text{ }x^{\ast }(t)\in \partial _{\ell
			(t)}f_{t}(x)\text{ ae}\right\} +N_{ \dom I_{f}}^{\epsilon _{2}}(x).
	\end{equation*}%
	If, in addition, $X$ is reflexive, then 
	\begin{equation*}
		\partial _{\epsilon }I_{f}(x)=\bigcup\limits_{\substack{ \epsilon =\epsilon
				_{1}+\epsilon _{2}  \\ \epsilon _{1},\epsilon _{2}\geq 0  \\ \ell \in 
				\mathcal{I}(\epsilon _{1})}}\int\limits_{T}\partial _{\ell (t)}f_{t}(x)d\mu
		(t)+N_{ \dom I_{f}}^{\epsilon _{2}}(x).
	\end{equation*}
\end{theorem}

\begin{proof}
	First, observe that the current continuity assumption of $\hat{I}_{f}$
	implies that both $I_{f}$ and the functions $f_{t},$ for ae $t\in T,$ are
	continuous at $x_{0}.$ Then, according to  \cref{Exactintegralformula}, to prove the first part\ we only need to verify\
	that, for every $x\in  \dom I_{f},$ $\epsilon _{1}\geq 0$ and $\ell
	\in \mathcal{I}(\epsilon _{1}),$ 
	\begin{equation}
		(w)\text{-}\int\limits_{T}\partial _{\ell (t)}f_{t}(x)d\mu (t)\subset \left\{
		\int\limits_{T}x^{\ast }d\mu :x^{\ast }\in L_{w^{\ast }}^{1}(T,X^{\ast }),%
		\text{ }x^{\ast }(t)\in \partial _{\ell (t)}f_{t}(x)\text{ ae}\right\} .
		\label{inc}
	\end{equation}%
	Take $x^{\ast }:=\int_{T}x^{\ast }d\mu $ for a $w^{\ast }$-measurable\
	function $x^{\ast }(\cdot )$ such that $x^{\ast }(t)\in \partial _{\ell
		(t)}f_{t}(x)$ ae. By  \cite[Theorem III.22.]{MR0467310}, for each
	function $\alpha \in L^{1}(T,(0,+\infty))$ there exists a measurable
	function $x:T\rightarrow X$ such that $\left\Vert x(t)\right\Vert \leq 1$
	and 
	\begin{equation*}
		\left\langle x^{\ast }(t),x(t)\right\rangle \geq \left\Vert x^{\ast
		}(t)\right\Vert -\alpha (t)\text{ ae;}
	\end{equation*}%
	hence, $\int_{T}\left\Vert x^{\ast }(t)\right\Vert d\mu (t)\leq
	\int_{T}\left\langle x^{\ast }(t),x(t)\right\rangle d\mu (t)+\int_{T}\alpha
	(t)d\mu (t).$ Since, by the continuity of $\hat{I}_{f}$ at $x_{0}$ there are 
	$\delta ,M>0$ such that 
	\begin{eqnarray*}
		\delta \int_{T}\left\langle x^{\ast }(t),x(t)\right\rangle d\mu (t) &\leq
		&\int_{T}f_{t}(\delta x(t)+x_{0})d\mu (t)-I_{f}(x)+\int_{T}\left\langle
		x^{\ast }(t),x-x_{0}\right\rangle d\mu (t)+\varepsilon _{1} \\
		&\leq &I_{f}(x_{0})-I_{f}(x)+1+\int_{T}\left\langle x^{\ast
		}(t),x-x_{0}\right\rangle d\mu (t)+\varepsilon _{1}\leq M,
	\end{eqnarray*}%
	we obtain that $\int_{T}\left\Vert x^{\ast }(t)\right\Vert d\mu (t)\leq
	\delta ^{-1}M+\int_{T}\alpha (t)d\mu (t)<+\infty ,$ and \cref{inc} holds.
	
	Finally, the last statement follows because $L_{w^{\ast }}^{1}(T,X^{\ast
	})=L^{1}(T,X^{\ast }).$\ 
\end{proof}

The next example shows that the second\ formula\ of  \cref{Banachintegralformula} cannot be valid if we drop the continuity of $%
\hat{I}_{f}.$

\begin{example}
	Consider $(T,\Sigma )=(\mathbb{N},\mathcal{P}(\mathbb{N}))$ and $X=\ell
	^{2}, $ and let\ $(e_{n})_{n}$\ be the canonical basis of $\ell ^{2},$ and $%
	\mu $ be the finite measure given by\ $\mu (\{n\})=2^{-n}.$ Define the
	integrand $f:\mathbb{N}\times \ell ^{2}\rightarrow \mathbb{R}$ as $%
	f(n,x):=2^{n}(\langle e_{n},x\rangle )^{2}$, so that $I_{f}(x)=\int_{\mathbb{%
			N}}f(n,x)d\mu (n)=\sum\limits_{n\in \mathbb{N}}x_{n}^{2}=\Vert x\Vert ^{2}$.
	Then $I_{f}$ is differentiable on $X$ with $\nabla I_{f}(x)=2x$, and for all 
	$n\geq 1$ we have that $\partial f_{n}(x)=\{\nabla
	f_{n}(x)\}=\{2^{n+1}\langle x,e_{n}\rangle e_{n}\},$ so that 
	\begin{equation*}
		\nabla I_{f}(x)=\int_{\mathbb{N}}2^{n+1}\langle x,e_{n}\rangle e_{n}d\mu
		(n)=2\sum_{n\in \mathbb{N}}\langle x,e_{n}\rangle e_{n}=2x,
	\end{equation*}%
	which is the result of  \cref{Exactintegralformula}. On the other
	side, the value $$\int_{n\in \mathbb{N}}\Vert \nabla f_{n}(x)\Vert d\mu
	(n)=2\sum_{n\in \mathbb{N}}|\langle x,e_{n}\rangle |$$ could not be finite
	for all $x\in \ell ^{2}$ (consider, for instance, $x=(1/n)_{n\geq 1}),$
	which means that $(2^{n+1}\langle x,e_{n}\rangle e_{n})_{n}\not\in L^{1}(%
	\mathbb{N},\ell ^{2}).$
\end{example}

The following is an easy consequence of  \cref{corollaryAsplund,Exactintegralformula}:

\begin{corollary}
	\label{CorGateauxdiff}Assume that either $X$ is Asplund, $X\ $and\ $X^{\ast
	} $ are Suslin, or $(T,\Sigma )=(\mathbb{N},\mathcal{P}(\mathbb{N})).$ If $%
	x\in X$ is a common continuity point of both $I_{f}$ and the $f_{t}$'s, then 
	$I_{f}$ is G\^{a}teaux-differentiable at $x$ if and only if $f_{t}$ is G\^{a}%
	teaux-differentiable at $x$ for ae $t\in T.$
\end{corollary}

The last result was given \cite[Corollary 2.11]{MR3571567} when $(T,\Sigma
)=(\mathbb{N},\mathcal{P}(\mathbb{N})).$ Concerning the Fr\'echet-differentiability, in the same referred result the authors proved one
implication (the Fr\'echet-differentiability of the sum implies the one of the
data functions), and  left the other implication as an open problem (see \cite[%
Question 2.12, page 1146]{MR3571567}). The following example answers this
question in the negative.

\begin{example}
	Consider $(T,\Sigma )=(\mathbb{N},\mathcal{P}(%
	\mathbb{N}))$ and $X=\ell ^{1},$ and let\ $(e_{n})_{n}$\ be the canonical
	basis of $\ell ^{1},$ and $\mu $ be the finite measure given by\ $\mu
	(\{n\})=1.$ Define the integrand $f:\mathbb{N}\times \ell ^{1}\rightarrow 
	\mathbb{R}$ as $f(n,x):=|\langle e_{n},x\rangle |^{1+1/n},$ so that $%
	I_{f}(x)=\sum |\langle e_{n},x\rangle |^{1+1/n}<+\infty .$ Since each $f_{n}$
	is a Fr\'echet-differentiable convex function such that\ $\nabla f_{n}(x)=(1+%
	\frac{1}{n})|\langle x,e_{n}\rangle |^{1/n}e_{n}$, according to 
	Corollary \ref{CorGateauxdiff}, $I_{f}$ is G\^{a}teaux-differentiable on $\ell ^{1},$ with\
	G\^{a}teaux-differential equal to $\sum \nabla f_{n}(x):=\int_{\mathbb{N}}\nabla
	f_{n}(x)d\mu (n)=\sum (1+\frac{1}{n})|\langle x,e_{n}\rangle |^{1/n}e_{n}$\
	(by  Corrollary \ref{CorGateauxdiff}). Thus, if\ $I_{f}$ were be\ Fr\'e%
	chet-differentiable at $x=0$, then\ we would have 
	\begin{equation*}
		\frac{I_{f}(n^{-1}e_{n})-I_{f}(0)-n^{-1}\langle \nabla I_{f}(0),e_{n}\rangle 
		}{n^{-1}}=nn^{-1-\frac{1}{n}}=n^{-\frac{1}{n}}\rightarrow 1,
	\end{equation*}%
	which is a contradiction.
\end{example}

In the following, we extend \cref{corollaryAsplund,Exactintegralformula}  to the nonconvex Lipschitz case. The resulting
formulae are known for both the case of a separable Banach space or $%
(T,\Sigma )=(\mathbb{N},\mathcal{P}(\mathbb{N}))$ (\cite[Theorem 2.7.2]{MR1058436}), and the case of Asplund spaces (\cite{mordukhovich2015subdifferentials}). Recall that for a Lipschtiz continuous
function $\varphi :X\rightarrow \mathbb{R}$, when $X$ is a normed space, the
generalized directional derivative of $\varphi $ at $x\in X$ in the
direction $u\in X$ is given by 
\begin{equation*}
	\varphi ^{\circ }(x;u):=\limsup\limits_{y\rightarrow x,\text{ }s\rightarrow
		0^{+}}s^{-1}(\varphi (y+su)-\varphi (y)).
\end{equation*}%
The generalized subdifferential of $\varphi $ at $x\in X$ is the set $%
\partial _{C}\varphi (x)$ defined as 
\begin{equation*}
	\partial _{C}\varphi (x):=\{x^{\ast }\in X^{\ast }:\varphi ^{\circ
	}(x;u)\geq \left\langle x^{\ast },u\right\rangle \text{ for all }u\in X\}.
\end{equation*}

\begin{proposition}[Clarke-Mordukhovich-Sagara]
	Assume that either\ $X$ is Asplund, $X$ is a separable Banach space, or $%
	(T,\Sigma )=(\mathbb{N},\mathcal{P}(\mathbb{N})).$ Let\ integrand $f:T\times
	X\rightarrow \overline{\mathbb{R}}$ and $x\in X$ be such that:
	
	\begin{enumerate}[label={(\alph*)},ref={(\alph*)}]
		
		\item There exists $K\in L^{1}(T,\mathbb{R}_{+})$ and $\delta >0$ such that
		for every $y,z\in B(x,\delta )$, $t\rightarrow f(t,y)$ is measurable and $%
		|f_{t}(y)-f_{t}(z)|\leq K(t)\Vert y-z\Vert $ ae $t\in T.$
		
		\item (When $X$ is Asplund) For every $u\in X$, the function $t\rightarrow
		f_{t}^{\circ }(x;u)$ is measurable.\newline
		Then we have that 
		\begin{eqnarray*}
			\partial _{C}I_{f}(x) &\subseteq & \cl \nolimits^{w^{\ast }}\left(
			(w)\text{-}\int\limits_{T}\partial _{C}f_{t}(x)d\mu (t)\right) \\
			&=&(w)\text{-}\int\limits_{T}\partial _{C}f_{t}(x)d\mu (t)\text{ (if }X%
			\text{ is separable or if }(T,\Sigma )=(\mathbb{N},\mathcal{P}(\mathbb{N}))%
			\text{)}.
		\end{eqnarray*}
	\end{enumerate}
\end{proposition}

\begin{proof}
	By taking into account Fatou's lemma, we have that $$I_{f}^{\circ }(x;u)\leq
	\int_{T}f_{t}^{\circ }(x;u)d\mu (t)$$ for all $u\in X,$ and so $\partial
	_{C}I_{f}(x)=\partial I_{f}^{\circ }(x;0)\subset \partial I_{f^{\circ
		}(x;\cdot )}(0).$\ Since\ $(t,u)\rightarrow f_{t}^{\circ }(x;u)$ is a Carath\'e%
	odory map, for every finite-dimensional subspace $F\subset X,$ the mapping $%
	(f_{\cdot }^{\circ }(x;\cdot ))_{|_{F}}:T\times F\rightarrow \mathbb{R\cup
		\{+\infty \}}$ is a convex normal integrand. Hence, because\ $I_{f^{\circ
		}(x;\cdot )}$ and $f_{t}^{\circ }(x;\cdot )$ are continuous everywhere, the
	two desired formulae follow by applying  \cref{corollaryAsplund,Exactintegralformula}, respectively.
\end{proof}

\section{Conjugate functions}\label{CONJ}

We investigate in this section the representation of the $\epsilon $-normal set to $%
\dom I_{f}$ in terms of the data functions $f_{t}$. 

We suppose that $%
f:T\times X\rightarrow \overline{\mathbb{R}}$ is a normal integrand defined
on a locally convex space$\,X$ such that for some\ $x_{0}^{\ast }\in 
\mathnormal{L}_{w^{\ast }}^{1}(T,X^{\ast })$ and $\alpha \in \mathnormal{L}%
^{1}(T,\mathbb{R})$ it holds\ 
\begin{equation}
	f(t,x)\geq \langle x_{0}^{\ast }(t),x\rangle +\alpha (t)\text{ for all }x\in
	X\text{ and }t\in T.  \label{growth}
\end{equation}
In what follows, we suppose that either $X,X^{\ast }$ are Suslin or $%
(T,\Sigma )=(\mathbb{N},\mathcal{P}(\mathbb{N})).$ We recall that the
continuous infimal convolution of the $f_{t}^{\ast }$'s is the function $$%
\oint\limits_{T} f^{\ast }(t,\cdot )d\mu (t):X^{\ast }\rightarrow 
\overline{\mathbb{R}}$$ given by (see \cite{MR1321585} ) 
\begin{eqnarray*}
	\left( \oint\limits_{T} f^{\ast }(t,\cdot )d\mu (t)\right) (x^{\ast })
	&:&=\oint\limits_{T} f^{\ast }(t,x^{\ast })d\mu (t) \\
	&:&=\inf \{\int\limits_{T}f^{\ast }(t,x^{\ast }(t))d\mu (t)\mid \begin{array}{c}
		x^{\ast
		}(\cdot )\in \mathnormal{L}_{w^{\ast }}^{1}(T,X^{\ast }) \\
		\text{ and }%
		\int\limits_{T}x^{\ast }(t)d\mu (t)=x^{\ast }
	\end{array} \},
\end{eqnarray*}%
with the convention that $\inf_{\emptyset }:=+\infty .$ We also recall the
notation $ \cl^{\tau}h$ and $ \cco^{\tau} h$ ($\cl h$ and $ \cco h$ when there is no confusion), which refers to the closure and the closed convex hull with respect to the topology $\tau$ on $Y$ of a function $h:Y \rightarrow \mathbb{R}\cup
\{+\infty \}.$ We shall need the following lemma, which can be found in \cite%
[Lemma 1.1]{MR1330645}.

\begin{lemma}
	\label{LEMMA:HIRIART}Let $h:X^{\ast }\rightarrow \overline{\mathbb{R}}$ be a
	convex function.\ Then for all $r\in \mathbb{R}$ 
	\begin{equation*}
		\{x^{\ast }\in X^{\ast }: \cl\nolimits^{w^{\ast }} h (x^{\ast })\leq
		r\}=\bigcap\limits_{\delta >0} \cl\{x^{\ast }\in X^{\ast }:h(x^{\ast
		})<r+\delta \}.
	\end{equation*}%
	Moreover, if $r>\inf_{X^{\ast }}h$, then 
	\begin{equation*}
		\{x^{\ast }\in X^{\ast }: \cl\nolimits^{w^{\ast }} h (x^{\ast })\leq
		r\}= \cl\{x^{\ast }\in X^{\ast }:h(x^{\ast })<r\}.
	\end{equation*}
\end{lemma}

\begin{theorem}
	\label{teorema3.9}If\ $\overline{ \co }I_{f}=I_{\overline{ \co %
		}f},$ then 
	\begin{equation*}
		(I_{f})^{\ast }(x^{\ast })= \cl\nolimits^{w^{\ast }}\left( \oint\limits_{T} f^{\ast }(t,\cdot )d\mu (t)\right) (x^{\ast }),\text{ for all 
		}x^{\ast }\in X^{\ast },
	\end{equation*}%
	and, for all $x\in X$ and $\epsilon \geq 0$, 
	\begin{equation*}
		\partial _{\epsilon }I_{f}(x)=\bigcap\limits_{\epsilon _{1}>\epsilon }%
		\cl\nolimits^{w^{\ast }}\left( \bigcup\limits_{\ell \in \mathcal{I}%
			(\epsilon _{1})}\left\{ \int\limits_{T}x^{\ast }d\mu :x^{\ast }\in
		L_{w^{\ast }}^{1}(T,X^{\ast }),\text{ }x^{\ast }(t)\in \partial _{\ell
			(t)}f_{t}(x)\text{ ae}\right\} \right) .
	\end{equation*}%
	In addition, if $\epsilon >I_{f}(x)-\overline{ \co }I_{f}(x)$, then 
	\begin{equation*}
		\partial _{\epsilon }I_{f}(x)= \cl\nolimits^{w^{\ast }}\left(
		\bigcup\limits_{\ell \in \mathcal{I}(\epsilon )}\left\{
		\int\limits_{T}x^{\ast }d\mu :x^{\ast }\in L_{w^{\ast }}^{1}(T,X^{\ast }),%
		\text{ }x^{\ast }(t)\in \partial _{\ell (t)}f_{t}(x)\text{ ae}\right\}
		\right) .
	\end{equation*}
\end{theorem}

\begin{proof}
	It easy to see  that $f^{\ast }$ is a convex integrand function\ such
	that $f^{\ast }(t,x_{0}^{\ast }(t))\leq -\alpha (t)$ ae, by \cref{growth}.
	In addition, denoting $\varphi :=\oint\limits_{T} f^{\ast }(t,\cdot )d\mu
	(t),$ we verify that, for all $x\in X,$\ 
	\begin{eqnarray*}
		\varphi ^{\ast }(x) &=&\sup\limits_{\lambda ^{\ast }\in X^{\ast
		}}\sup\limits_{x^{\ast }\in \mathnormal{L}_{w^{\ast }}^{1}(T,X^{\ast }),%
			\text{ }\int\limits_{T}x^{\ast }d\mu =\lambda ^{\ast }}\{\langle \lambda
		^{\ast },x\rangle -\int\limits_{T}f^{\ast }(t,x^{\ast }(t))d\mu (t)\} \\
		&=&\sup\limits_{x^{\ast }\in \mathnormal{L}_{w^{\ast }}^{1}(T,X^{\ast
			})}\int\limits_{T}(\langle x^{\ast }(t),x\rangle -f^{\ast }(t,x^{\ast
		}(t))d\mu (t),
	\end{eqnarray*}%
	and so, according to \cite[Proposition 3.3]{INTECONV} (see also \cite[Theorem VII-7]{MR0467310}) and Moreau's envelope
	Theorem, 
	\begin{equation*}
		\varphi ^{\ast }(x)=I_{f^{\ast \ast }}(x)=I_{\overline{ \co }f}(x)=%
		\overline{ \co }I_{f}.
	\end{equation*}%
	Consequently, the convexity of the continuous infimal convolution yields 
	\begin{equation}
		\cl^{w^{\ast }}\left( \varphi \right) =\varphi ^{\ast \ast
		}=(\overline{ \co }I_{f})^{\ast }=(I_{f})^{\ast }.
		\label{equationbiconjugate}
	\end{equation}
	
	Now, we\ assume that $\partial _{\epsilon }I_{f}(x)\neq \emptyset $. Then,
	for all\ $\epsilon _{1}>\epsilon $ one has that $$\inf_{x^{\ast }\in X^{\ast
	}}\{\varphi (x^{\ast })-\langle x^{\ast },x\rangle +I_{f}(x)\}=I_{f}(x)-(%
	\overline{ \co }I_{f})(x)\leq \epsilon <\epsilon _{1}$$ (by \cref%
	{equationbiconjugate}), and so, using  Lemma \ref{LEMMA:HIRIART}, 
	\begin{eqnarray*}
		\partial _{\epsilon _{1}}I_{f}(x) &=&\left\{ x^{\ast }\in X^{\ast }:{%
			\cl}^{w^{\ast }}\left( \varphi \right) (x^{\ast })+I_{f}(x)-\langle
		x^{\ast },x\rangle \leq \epsilon _{1}\right\} \\
		&=&\left\{ x^{\ast }\in X^{\ast }: \cl \nolimits^{w^{\ast }}\left(
		\varphi -x\right) (x^{\ast })+I_{f}(x)\leq \epsilon _{1}\right\} \\
		&=& \cl \nolimits^{w^{\ast }}\left\{ x^{\ast }\in X^{\ast }:\varphi
		(x^{\ast })+I_{f}(x)<\langle x^{\ast },x\rangle +\epsilon _{1}\right\} .
	\end{eqnarray*}%
	Observe that for each\ $x^{\ast }\in X^{\ast }$ satisfying\ $\varphi
	(x^{\ast })+I_{f}(x)<\langle x^{\ast },x\rangle +\epsilon _{1}$, there
	exists some $x^{\ast }(\cdot )\in \mathnormal{L}_{w^{\ast }}^{1}(T,X^{\ast
	}) $ such that $\int_{T}f(t,x)d\mu (t)+\int f^{\ast }(t,x^{\ast }(t))\leq
	\int \langle x^{\ast }(t),x\rangle d\mu (t)+\epsilon _{1}$. Hence, by
	defining $\ell (t):=f(t,x)+f^{\ast }(t,x^{\ast }(t))-\langle x^{\ast
	}(t),x\rangle ;$ hence, $\ell \in \mathcal{I}(\epsilon _{1})\ $and\ $x^{\ast
	}(t)\in \partial _{\ell (t)}f(t,x)$ ae. The proof is finished since the
	other inclusion is straightforward.
\end{proof}

\begin{remark}
It is worth observing that if the linear growth condition 
	\cref{growth} holds with $x_{0}^{\ast }\in \mathnormal{L}_{w^{\ast
	}}^{p}(T,X^{\ast })$ (or $\mathnormal{L}^{p}(T,X^{\ast }),$ resp.) for some $%
	p\in \lbrack 1,+\infty ]$, then changing the infimum in the definition of $\oint\limits_{T} f^*(t,\cdot) d\mu(t)$ over elements $	x^{\ast
	}(\cdot ) \in  \mathnormal{L}_{w^{\ast
	}}^{p}(T,X^{\ast })$  (or $\mathnormal{L}^{p}(T,X^{\ast }),$ resp.)  we obtain more precision on the
	subdifferenital set of $I_{f},$ namely, for every $x\in X$%
	\begin{equation*}
		\partial _{\epsilon }I_{f}(x)=\bigcap\limits_{\epsilon _{1}>\epsilon }%
		\cl \nolimits^{w^{\ast }}\left( \left\{ \int\limits_{T}x^{\ast }d\mu :
		\begin{array}{c}
			x^{\ast }\in
			L_{w^{\ast }}^{p}(T,X^{\ast })\text{ }(x^{\ast }\in \mathnormal{L}%
			^{p}(T,X^{\ast })\text{, resp.}),\\\ell \in \mathcal{I}%
			(\epsilon _{1}) \text{ and }x^{\ast }(t)\in \partial _{\ell
				(t)}f_{t}(x)\text{ ae}
		\end{array}
		\right\} \right) ,
	\end{equation*}%
	and similarly for the case $\epsilon >I_{f}(x)-\overline{ \co }I_{f}$. The same holds true if the space of $p$-integrable functions is replaced by any decomposable space (see e.g. \cite[Definition 3, \S VII]{MR0467310}, or \cite[Definition 3.2]{INTECONV} for this definition).
\end{remark}

We also obtain a characterization of the epigraph of the function $%
I_{f}^{\ast }:$
\begin{corollary}
	\label{corollary:conjugate}Assume that $I_{f}^{\ast }$ is proper. If\ $%
	\overline{ \co }I_{f}=I_{\overline{ \co }f},$ then we have
	that 
	\begin{equation*}
		\epi I_{f}^{\ast }={\cl}^{w^{\ast }}\left( \left\{ \left(
		\int\limits_{T}x^{\ast }d\mu ,\int_{T}\alpha d\mu \right) :\begin{array}{c}
			x^{\ast }\in
			L_{w^{\ast }}^{1}(T,X^{\ast }),\; \alpha \in L^{1}(T,\mathbb{R}),\\(x^{\ast }(t),\alpha (t))\in \epi f_{t}^{\ast }\text{ ae}
		\end{array}\right\} \right) .
	\end{equation*}
\end{corollary}

\begin{proof}
	We denote $E$ the set between parentheses in\ the equation above. Take $%
	(x^{\ast },\alpha )\in E.$ Then, using again the notation $\varphi :=\oint\limits_{T}f^{\ast }(t,\cdot )d\mu (t),$\ we obtain that $\varphi
	(x^{\ast })\leq \alpha ,$ and so by  \cref{teorema3.9} we have that $(x^{\ast
	},\alpha )\in \epi I_{f}^{\ast }.$ Hence, the lower semicontinuity of $%
	I_{f}^{\ast }$ yields the inclusion $ \cl \nolimits^{w^{\ast
	}}(E)\subset \epi I_{f}^{\ast }$. To prove the other inclusion,\ we take $%
	(x^{\ast },\alpha )\in \epi I_{f}^{\ast }$, and\ fix $\epsilon >0$ and $V\in 
	\mathcal{N}_{x^{\ast }}(w^{\ast })$ together with $\gamma (\cdot )\in
	L^{1}(T,\mathbb{R}_{+})$ such that $\int_{T}\gamma d\mu =1.$ Then by 
	\cref{teorema3.9} there exists $x^{\ast }(\cdot )\in L_{w^{\ast
	}}^{1}(T,X^{\ast })$ such that (w.l.o.g.) $\int\limits_{T}x^{\ast }d\mu \in
	V $ and 
	\begin{equation*}
		-\infty <(I_{f})^{\ast }(x^{\ast })-1= \cl \nolimits^{w^{\ast
		}}\left( \varphi \right) (x^{\ast })-1\leq \varphi \left(
		\int\limits_{T}x^{\ast }d\mu \right) \leq \int\limits_{T}f^{\ast }(t,x^{\ast
		}(t))d\mu (t)\leq \alpha +\epsilon.
	\end{equation*}%
	Thus, if we denote\ $\beta (t):=f^{\ast }(t,x^{\ast }(t))+\gamma (t)\left(
	\alpha +\epsilon -\int\limits_{T}f^{\ast }(t,x^{\ast }(t))d\mu (t)\right) $,
	we get  $\int_{T}\beta d\mu =\alpha +\epsilon $ and so $(x^{\ast }(t),\beta
	(t))\in \epi f_{t}$ and $(\int\limits_{T}x^{\ast }d\mu ,\int\limits_{T}\beta
	d\mu )\in E.$ Hence, from the arbitrariness of $\epsilon >0$ and $V$ we
	deduce that $(x^{\ast },\alpha )\in  \cl \nolimits^{w^{\ast }}(E).$
\end{proof}
We are now in position to give the desired representation of the $\epsilon $%
-normal set to ${\dom}I_{f}.$
\begin{proposition}\label{lemma:normalcone}
	Assume that $f$ is a convex normal integrand. Then
	for every $x\in \dom I_{f}$ and $\epsilon \geq 0$ we have that 
	\begin{align*}
		N_{\dom I_{f}}^{\epsilon }(x)& =\{x^{\ast }\in X^{\ast }:(x^{\ast
		},\langle x^{\ast },x\rangle +\epsilon )\in \epi(\sigma _{\dom %
			I_{f}})\} \\
		& =\{x^{\ast }\in X^{\ast }:(x^{\ast },\langle x^{\ast },x\rangle +\epsilon
		)\in (\epi(I_{f})^{\ast })_{\infty }\} \\
		& =\left\{ x^{\ast }\in X^{\ast }:(x^{\ast },\langle x^{\ast },x\rangle
		+\epsilon )\in \left[ \cl  \nolimits^{w^{\ast }} \mathcal{E} \right] _{\infty }\right\}
		\\
		& =\left\{ x^{\ast }\in X^{\ast }:(x^{\ast },\langle x^{\ast },x\rangle
		+\epsilon )\in \left[ \overline{\co}^{w^{\ast }}\mathcal{G}\right] _{\infty
		}+\{0\}\times \lbrack 0,\epsilon ]\right\},
	\end{align*}
	where \begin{align}
		\mathcal{E}:=&\left\{ \big(
		\int\limits_{T}x^{\ast }d\mu ,\int_{T}\alpha d\mu \big) : \begin{array}{c} 
			x^{\ast }\in \textnormal{L}_{w^{\ast }}^{1}(T,X^{\ast }),\;
			\alpha \in \textnormal{L}^{1}(T,\mathbb{R}),\\
			(x^{\ast }(t),\alpha (t))\in \epi f_{t}^{\ast }\text{ ae}
		\end{array} \right\}\label{epigrafo1}, \\
		\mathcal{G}:=&\left\{ \big(
		\int\limits_{T}x^{\ast }d\mu ,\int_{T}\alpha d\mu \big) : \begin{array}{c} 
			x^{\ast }\in \textnormal{L}_{w^{\ast }}^{1}(T,X^{\ast }),\;	\alpha \in \textnormal{L}^{1}(T,\mathbb{R}),\\
			(x^{\ast }(t),\alpha (t))\in \grafo f_{t}^{\ast }\text{ ae}
		\end{array} \right\}. \label{epigrafo2}
	\end{align}
\end{proposition}

\begin{proof}
	For the first two equalities see \cite[Lemma 5]{MR2448918}. The third
	one is given by  Corollary \ref{corollary:conjugate}. So, we only have to
	prove the fourth equality, or equivalently, the inclusion "$\subseteq ".$ On
	the one hand, we have that 
	\begin{align}
		\cl  \nolimits^{w^{\ast }}\mathcal{E} &=\cl  \nolimits^{w^{\ast }}\left( \cco^{w^{\ast }}\mathcal{G}  +\{0\}\times \mathbb{R}_{+}\right) .  \label{so}
	\end{align}%
	Indeed, to see the last inclusion, take $(x^{\ast },\alpha )\in \mathcal{E}$
	and let $(x^{\ast }(t),\alpha (t))\in \epi f_t^\ast = \grafo f_{t}^{\ast }+\{0\}\times 
	\mathbb{R}_{+}$ ae such that $x^{\ast }(\cdot )\in \textnormal{L}_{w^{\ast }}^{1}(T,X^{\ast
	}),$ $\alpha (\cdot )\in \textnormal{L}^{1}(T,\mathbb{R}),$ and $(x^{\ast },\alpha
	)=(\int_{T}x^{\ast }d\mu ,\int_{T}\alpha d\mu ).$ Then, since $(I_{f})^{\ast
	}$ is proper, we have that $\int_{T}f_{t}^{\ast }(x^{\ast }(t))d\mu (t)\in 
	\mathbb{R}$ and so, writing 
	\begin{equation*}
		(x^{\ast }(t),\alpha (t))=(x^{\ast }(t),f_{t}^{\ast }(x^{\ast
		}(t)))+(0,\alpha (t)-f_{t}^{\ast }(x^{\ast }(t)))\in \grafo f_{t}^{\ast
		}+\{0\}\times \mathbb{R}_{+},
	\end{equation*}%
	we get that $(x^{\ast },\alpha )\in \mathcal{G}+\{0\}\times \mathbb{R}_{+}$, besides by the convexity of $f_t^\ast$'s $\mathcal{G}+\{0\}\times \mathbb{R}_{+} \subseteq \mathcal{E}$.
	
	On the other hand, since ($(I_{f})^{\ast }$ is proper) we have that 
	\begin{align*}
		\left[ \overline{\co}^{w^{\ast }}\mathcal{G} \right]
		_{\infty }\cap \left( -\left[ \{0\}\times \mathbb{R}_{+}\right] _{\infty
		}\right) &\subseteq \left[ \cl  \nolimits^{w^{\ast }}\mathcal{E} \right] _{\infty }\cap
		\left( \{0\}\times \mathbb{R}_{-}\right) \\&=\left( \epi(I_{f})^{\ast }\right)
		_{\infty }\cap \left( \{0\}\times \mathbb{R}_{-}\right) =\{(0,0)\},
	\end{align*}%
	and so by Dieudonn\'e's Theorem (see \cite[Thorem I-10]{MR0467310} or \cite[%
	Proposition 1]{MR0194865}) the set $\overline{\co}^{w^{\ast
	}}\mathcal{G} +\{0\}\times \mathbb{R}_{+}$ is closed. Hence,\ \cref{so}  reads 
	\begin{align*}
		\left[ \cl  \nolimits^{w^{\ast }} \mathcal{E} \right] _{\infty }=\left[ \overline{%
			\co}^{w^{\ast }}\mathcal{G}  +\{0\}\times \mathbb{R}_{+}\right]
		_{\infty }=\left[ \overline{\co}^{w^{\ast }}\mathcal{G} \right]
		_{\infty }+\{0\}\times \mathbb{R}_{+}.
	\end{align*}%
	Now, we take $x^{\ast }\in X^{\ast }$ such that $(x^{\ast },\langle x^{\ast
	},x\rangle +\epsilon )\in \left[ \cl  \nolimits^{w^{\ast }}\mathcal{E} \right] _{\infty }$%
	. Then, by the last relations, there exist $(y^{\ast },\gamma )\in \left[ 
	\overline{\co}^{w^{\ast }}\mathcal{G} \right] _{\infty }$ and $\eta
	\geq 0$ such that $(x^{\ast },\langle x^{\ast },x\rangle +\epsilon
	)=(y^{\ast },\gamma +\eta )$; hence, $x^{\ast }=y^{\ast }$. Moreover, using
	\cref{teorema3.9}, we have 
	\begin{align*}
		\dom I_{f}\times \{-1\}\subseteq \left[ (\epi(I_{f})^{\ast
		})_{\infty }\right] ^{\circ }=\left[  \big(\cl  \nolimits^{w^{\ast
		}}\mathcal{E}\big)
		_{\infty }\right] ^{\circ }\subseteq \left[ \big( \cl  %
		\nolimits^{w^{\ast }} \mathcal{G} \big) _{\infty }\right] ^{\circ },
	\end{align*}%
	so that $\langle (x^{\ast },\gamma ),(x,-1)\rangle \leq 0$, and $\eta
	=\langle x^{\ast },x\rangle -\gamma +\epsilon \leq \epsilon $; that is, $$
	(x^{\ast },\langle x^{\ast },x\rangle )\in \left[ \overline{\co}%
	^{w^{\ast }} \big( \mathcal{G} \big) \right] _{\infty }+\{0\}\times \lbrack 0,\epsilon ].$$
\end{proof}

Now, we obtain a complete\ explicit characterization of the $\epsilon\text{-}$subdifferential of $I_{f}$ in  terms of the nominal data $f_t'$s.
\begin{theorem}
Assume that $f$ is a convex normal integrand. Then
	for every $x\in X$ and $\varepsilon \geq 0$ we have that 
	\begin{align*}
		\partial _{\epsilon }I_{f}(x)& =\bigcap\limits_{L\in \mathcal{F}%
			(x)}\bigcup\limits_{\substack{ \epsilon _{1},\epsilon _{2}\geq 0  \\ %
				\epsilon =\epsilon _{1}+\epsilon _{2}  \\ \ell \in \mathcal{I}(\epsilon
				_{1})  \\  \\ \eta \in \mathnormal{L}^{1}}}\bigcap\limits_{\eta \in 
			\mathnormal{L}^{1}(T,(0,+\infty ))}\cl  \left\{ \int\limits_{T}\left(
		\partial _{\ell (t)+\eta (t)}f_{t}(x)+A_{L}^{\epsilon _{2}}(x)\right) d\mu
		(t)\right\} \\
		& =\bigcap\limits_{L\in \mathcal{F}(x)}\bigcup\limits_{\substack{ \epsilon
				_{1},\epsilon _{2}\geq 0  \\ \epsilon =\epsilon _{1}+\epsilon _{2}  \\ \ell
				\in \mathcal{I}(\epsilon _{1})  \\ \eta \in \mathnormal{L}^{1}}}%
		\bigcap\limits_{\eta \in \mathnormal{L}^{1}(T,(0,+\infty ))}\cl  %
		\left\{ \int\limits_{T}\left( \partial _{\ell (t)+\eta
			(t)}f_{t}(x)+B_{L}^{\epsilon _{2}}(x)\right) d\mu (t)\right\} ,
	\end{align*}%
	where the closure is taken with respect to the strong topology $\beta (X^{\ast },X),$ 
	\begin{align*}
		A_{L}^{\epsilon _{2}}(x)&\hspace{-0.1cm} :=\left\{ x^{\ast }\in X^{\ast }:(x^{\ast
		},\langle x^{\ast },x\rangle +\epsilon _{2})\in \left[ \cl  %
		\nolimits^{w^{\ast }} \big(\mathcal{E}+L^{\perp }\times \mathbb{R}_{+} \big) \right] _{\infty }\right\} \\
		B_{L}^{\epsilon _{2}}(x)& \hspace{-0.1cm}:=\left\{ x^{\ast }\in X^{\ast }:(x^{\ast
		},\langle x^{\ast },x\rangle +\epsilon _{2})\in \left[ \cl^{w^\ast}\big(\co \big( \mathcal{G}\big) +  L^{\perp }\times \mathbb{R}_{+} \big)  \right] _{\infty }\hspace{-0.4cm}+\{0\}\times
		\lbrack 0,\epsilon _{2}]\right\},
	\end{align*}
	and $\mathcal{E}$ and $\mathcal{G}$ are defined in \cref{epigrafo1,epigrafo2}, respectively.
\end{theorem}

\begin{proof}
	According\ to \cite[Theorem 5.1]{INTECONV}, we only have to prove that for every 
	$L\in \mathcal{F}(x)$ and $\epsilon \geq 0$, $N_{ \dom I_{f}\cap
		L}^{\epsilon }(x)=A_{L}^{\epsilon }(x)=B_{L}^{\epsilon }(x)$. Indeed, it
	suffices to\ apply  Proposition \ref{lemma:normalcone} with the measurable
	space $(\tilde{T},\tilde{\Sigma },\tilde{\mu}),$ where $\tilde{T}%
	:=T\cup \{\omega _{0}\}$ for an element $\omega _{0}\notin T$, $\tilde{%
		\Sigma }$ is the $\sigma $-Algebra generated by $(\Sigma %
	\cup \{\omega _{0}\}),$ and $\tilde{\mu}$ is defined by 
	\begin{equation*}
		\tilde{\mu}(G):=\left\{ 
		\begin{array}{ll}
			\mu (G\backslash \{\omega _{0}\})+1 & \text{ if }\omega _{0}\in G \\ 
			\mu (G) & \text{ if }\omega _{0}\notin G,%
		\end{array}%
		\right.
	\end{equation*}%
	and the integrand function $g(t,x):=f(t,x)$ for $t\in T$ and $g(\omega
	_{0},x):=\delta _{L}(x).$
\end{proof}

\section{Characterizations via (exact-) subdifferentials\label{BronstedRockafellartheorems}}

In this section, we use the\ previous results\ and Br{o}nsted-Rockafellar
theorems to obtain sequential formulae for the subdifferential of integral
functions. As in the previous section, we suppose that either $X,X^{\ast }$
are Suslin or $(T,\Sigma )=(\mathbb{N},\mathcal{P}(\mathbb{N})).$

We recall that a net of weakly measurable functions $g_{i}:T\rightarrow X\ $%
is said to converge uniformly ae to $g:T\rightarrow X$ if for each continuous
seminorm $\rho $ in $X,$ the net $\rho (g_{i}-g)$ converges to $0$ in $%
L^{\infty }(T,\mathbb{R}).$

\begin{theorem}
	\label{Meanvalue}Suppose that one of the following conditions holds:
	
	\begin{enumerate}[label={(\roman*)},ref={(\roman*)}]
		
		\item \label{Meanvalue:I}$X$ is Banach.
		
		\item \label{Meanvalue:II}$f_{t}$ are epi-pointed ae\ $t$.
	\end{enumerate}
	
	Then for every $x\in X,$ we have that $x^{\ast }\in \partial I_{f}(x)$ if
	and only if there exist a net of finite-dimensional subspaces $(L_{i})_{i}$
	and nets of measurable selections $(x_{i}),$ $(x_{i}^{\ast })$ and $(y_{i}),$ $%
	(y_{i}^{\ast })$ such that $x_{i}^{\ast }(t)\in \partial f(t,x_{i}(t))$, $%
	y_{i}^{\ast }(t)\in N_{ \dom I_{f}\cap L_{i}}(y(t))$ ae, and:
	
	\begin{enumerate}[label={(\alph*)},ref={(\alph*)}]
		
		\item \label{Meanvaluea}$(x_{i}^{\ast }+y_{i}^{\ast })\subset
		L^{1}(T,X^{\ast })$ and $x^{\ast }={w^{\ast }}$\textnormal{-}$\lim
		\int\limits_{T}(x_{i}^{\ast }(t)+y_{i}^{\ast }(t))d\mu (t)$.
		
		\item \label{Meanvalueb}$x_{i},y_{i}\rightarrow x$ uniformly ae.
		
		\item \label{Meanvaluec}$f(\cdot ,x_{i}(\cdot ))\rightarrow f(\cdot ,x)$
		uniformly ae.
		
		\item \label{Meanvalued}$\langle x_{i}^{\ast }(\cdot ),x_{i}(\cdot
		)-x\rangle ,$ $\langle y_{i}^{\ast }(\cdot ),y_{i}(\cdot )-x\rangle
		\rightarrow 0$ uniformly ae.\newline
		In addition, if $X$ is reflexive and separable, then we take sequences
		instead of nets, and the $w^{\ast }$-convergence is replaced by the norm.
	\end{enumerate}
\end{theorem}

\begin{proof}
	W.l.o.g. we may assume that $\mu (T)<+\infty .$ Take\ $x_{0}^{\ast }\in
	\partial I_{f}(x_{0}),$ $x_{0}\in X,$ and fix $L\in \mathcal{F}(x_{0})$. By   \cite[Theorem 4.1]{INTECONV}  we find a measurable function $%
	z^{\ast }(\cdot )$\ such that $z^{\ast }(t)\in \partial (f_{t}+\delta
	_{L\cap  \dom I_{f}})(x_{0})$ for all $t\in \tilde{T}$ (with $\mu
	(T\backslash \tilde{T})=0$), and $x_{0}^{\ast }=\int_{T}z^{\ast }(t)d\mu
	(t). $ Next, given $n\in \mathbb{N}$, continuous seminorms $\rho _{X}$ on $X$
	and a $w^{\ast }$-continuous seminorm $\rho _{X^{\ast }}:=\sigma _{C}$ on $%
	X^{\ast }$ such that $C$ is finite and $ \spn C\supset  \spn %
	(L\cap  \dom I_{f}),$ we define the multifunction $B:\tilde{T}%
	\rightrightarrows X\times X^{\ast }\times X\times X^{\ast }$ by $(x,x^{\ast
	},y,y^{\ast })\in B(t)$ if and only if
	
	\begin{enumerate}[label={A(\roman*)},ref={A(\roman*)}]
		
		\item $x^{\ast }\in \partial f(t,x)$, $y^{\ast }\in N_{ \dom %
			I_{f}\cap L}(y)$.\label{m1}
		
		\item $\rho _{X}(x-x_{0})\leq 1/n$, $\rho _{X}(y-x_{0})\leq 1/n$, $\rho
		_{X^{\ast }}(z^{\ast }(t)-x^{\ast }-y^{\ast })\leq 1/n.$\label{m2}
		
		\item $|f(t,x)\rightarrow f(t,x_{0})|\leq 1/n$, $\left\vert \langle x^{\ast
		},x-x_{0}\rangle \right\vert \leq 1/n$ and $\left\vert \langle y^{\ast
		},y-x_{0}\rangle \right\vert \leq 1/n.$\label{m3}
	\end{enumerate}
	By \cite[Theorem 2.3]{MR1448052} (see, also, \cite[Theorem 3]{MR1358408})
	(in  \cref{Meanvalue:I}) or by \cite[Theorem 4.7]{BronstedRock} (in
	\cref{Meanvalue:II}), $B(t)$ is non-empty for all $t\in \tilde{T}.$
	Hence, due to the measurability of the involved functions, $B$ has a
	measurable graph, so that by  \cite[Theorem III.22.]{MR0467310}  we
	conclude the existence of nets of measurable functions $x(\cdot ),$ $y(\cdot
	),$ $x^{\ast }(\cdot ),$ and $\tilde{y}^{\ast }(\cdot ),$ which satisfy the
	properties \cref{m1,m2,m3}  above. Now, we consider\ $%
	y^{\ast }(t):=P_{ \spn C}^{\ast }(x^{\ast }(t))-x^{\ast }(t)+P_{%
		\spn C}^{\ast }(\tilde{y}^{\ast }(t)),$ where $P^{\ast }$ is the
	adjoint of the a projection $P_{ \spn C}$ onto $ \spn C.$
	Then 
	\begin{equation*}
		\left\langle y_{i}^{\ast }(t),u-y(t)\right\rangle =\left\langle \tilde{y}%
		^{\ast }(t),u-y(t)\right\rangle \leq 0\text{ \ }\forall u\in  \dom %
		I_{f}\cap L_{i}( \spn C),
	\end{equation*}%
	and so $y_{i}^{\ast }(t)\in N_{ \dom I_{f}\cap L}(y(t)).$ Moreover,
	we have 
	\begin{eqnarray*}
		\rho _{X^{\ast }}(z^{\ast }(t)-x^{\ast }(t)-y^{\ast }(t)) &=&\rho _{X^{\ast
		}}(z^{\ast }(t)-P_{ \spn C}^{\ast }(x^{\ast }(t))-P_{ \spn %
			C}^{\ast }(\tilde{y}^{\ast }(t))) \\
		&=&\sigma _{C}(z^{\ast }(t)-P_{ \spn C}^{\ast }(x^{\ast }(t))-P_{%
			\spn C}^{\ast }(\tilde{y}^{\ast }(t))) \\
		&=&\sigma _{C}(z^{\ast }(t)-x^{\ast }(t)-\tilde{y}^{\ast }(t))\leq 1/n,
	\end{eqnarray*}%
	and
	\begin{eqnarray*}
		\left\vert \langle y^{\ast }(t),y(t)-x_{0}\rangle \right\vert &=&\left\vert
		\langle P_{ \spn C}^{\ast }(x^{\ast }(t))-x^{\ast }(t)+P_{\spn C}^{\ast }(\tilde{y}^{\ast }(t)),y(t)-x_{0}\rangle \right\vert \\
		&=&\left\vert \langle \tilde{y}^{\ast }(t),y(t)-x_{0}\rangle \right\vert
		\leq 1/n,
	\end{eqnarray*}%
	and for every balanced bounded set $A\subset X$%
	\begin{align*}
		\int_{T}\sigma _{A}(x^{\ast }(t)+y^{\ast }(t))d\mu (t)=&\int_{T}\sigma
		_{A}(P_{ \spn C}^{\ast }(x^{\ast }(t))+P_{ \spn C}^{\ast }(%
		\tilde{y}^{\ast }(t)))d\mu (t) \\
		=&\int_{T}\sigma _{P(A)}(x^{\ast }(t)+\tilde{y}^{\ast }(t)-z^{\ast }(t))d\mu
		(t)\\&+\int_{T}\sigma _{P(A)}(z^{\ast }(t))d\mu (t)<+\infty .
	\end{align*}
	{The conclusion when }$X$ is reflexive and separable\ comes from the fact
	that we can take sequences instead of nets used above (using \cite[Theorem
	2.3]{MR1448052}) and countable family of finite-dimensional subspaces
	(see  \cite[Remark 4.3]{INTECONV}).
	
	It remains to verify the sufficiency implication. Take\ $x^{\ast }\in
	X^{\ast }$, $x\in X$, nets of finite-dimensional subspaces $L_{i}$ and nets
	of measurable functions\ $(x_{i}),$ $(x_{i}^{\ast })$ and $(y_{i}),$ $%
	(y_{i}^{\ast })$ as  in the statement of the theorem. Then\ for all $u\in
	X $ we obtain 
	\begin{align*}
		\langle x^{\ast },u-x\rangle & =\lim \int \langle x_{i}^{\ast }+y_{i}^{\ast
		},u-x\rangle d\mu \\
		& =\int\limits_{T}\left( \langle x_{i}^{\ast }(t),u-x_{i}(t)\rangle +\langle
		x_{i}^{\ast }(t),x_{i}(t)-x\rangle  \right)
		d\mu \\&+\int\limits_{T}\left( \langle y_{i}^{\ast
		}(t),y_{i}(t)-x\rangle +\langle y_{i}^{\ast }(t),u-y_{i}(t)\rangle \right)
		d\mu \\
		& \leq \lim \int\limits_{T}\left\{ f(t,u)-f(t,x_{i}(t))\right\} d\mu
		(t)+\lim \int\limits_{T}\left\{ \langle x_{i}^{\ast },x_{i}(t)-x\rangle
		\right\} d\mu (t) \\
		& +\lim \int\limits_{T}\left\{ \langle y_{i}^{\ast }(t),y_{i}(t)-x\rangle
		\right\} d\mu (t) \\
		& =I_{f}(u)-I_{f}(x);
	\end{align*}%
	that is, $x^{\ast }\in \partial I_{f}(x).$
\end{proof}

We shall need the following lemma.

\begin{lemma}
	\label{lemma:previo}Suppose that the linear growth condition \cref{growth}
	holds with some $x_{0}^{\ast }\in \mathnormal{L}_{w^{\ast }}^{1}(T,X^{\ast })$.
	If\ the $f_{t}$'s are epi-pointed ae $t\in T$ and $x_{0}^{\ast }(t)\in \inte {\dom}f_{t}^{\ast }$ ae $t\in T$, then 
	\begin{equation*}
		\partial I_{f}(x)=\bigcap\limits_{\varepsilon >0}{\cl}\left\{
		\bigcup\limits_{\ell \in \mathcal{I}(\varepsilon
			)}(w)\text{-}\int\limits_{T}\partial _{\ell (t)}f(t,x)\cap \inte({\dom}%
		f_{t}^{\ast })d\mu (t)\right\}.
	\end{equation*}
\end{lemma}

\begin{proof}
	Since the inclusion $\supseteq $ is trivial we focus on the opposite one.
	According to  \cref{teorema3.9}, it suffices\ to show that for every $%
	\varepsilon >0$, $\hat{\ell}\in \mathcal{I}(\varepsilon )$ and $z^{\ast }\in
	(w)\text{-}\int\limits_{T}\partial _{\hat{\ell}(t)}f(t,x)d\mu (t)$ we have that $$
	z^{\ast }\in  \cl  \left\{ \bigcup\limits_{\ell \in \mathcal{I}%
		(2\varepsilon )}(w)\text{-}\int\limits_{T}\partial _{\ell (t)}f(t,x)\cap \inte(%
	\dom f_{t}^{\ast })d\mu (t)\right\} .$$ Since $f^{\ast }(\cdot
	,z^{\ast }(\cdot ))\in \mathnormal{L}^{1}(T,\mathbb{R})$, we have that\ $%
	z^{\ast }(t)\in  \dom f_{t}^{\ast }$ for ae $t\in T,$ and so $%
	z_{\lambda }^{\ast }(t)=(1-\lambda )z^{\ast }(t)+\lambda x_{0}^{\ast }(t)\in %
	\inte( \dom f_{t}^{\ast })\cap \partial _{\ell _{\lambda
		}(t)}f_{t}(x) $ where $\lambda \in (0,1)$ and $\ell _{\lambda }(t):=\langle
	z_{\lambda }^{\ast },x\rangle -f(t,x)-f^{\ast }(t,z_{\lambda }^{\ast
	}(t))\geq 0.$ By the Fenchel inequality and convexity of the $f_{t}^{\ast }$%
	's we get 
	\begin{align*}
		\langle z_{\lambda }^{\ast }(t),x\rangle -f(t,x)-f^{\ast }(t,z^{\ast
		}(t))& \leq f^{\ast }(t,z_{\lambda }^{\ast }(t))-f^{\ast }(t,z^{\ast }(t))\\
		&\leq
		\lambda (f^{\ast }(t,x_{0}^{\ast }(t))-f^{\ast }(t,z^{\ast }(t))),
	\end{align*}%
	and so, since\ $f^{\ast }(t,z_{\lambda }^{\ast })\rightarrow f^{\ast
	}(t,z^{\ast }(t))$ as $\lambda \downarrow 0$, we have $\ell _{\lambda
	}(t)\rightarrow \langle z^{\ast },x\rangle -f(t,x)-f^{\ast }(t,z^{\ast
	}(t))\leq \hat{\ell}(t)$, by the Lebesgue dominated convergence theorem we\
	get $$\lim_{\lambda \rightarrow 0}\int_{T}\ell _{\lambda }(t)d\mu (t)\leq
	\int_{T}\hat{\ell}(t)d\mu (t)\leq \varepsilon .$$
\end{proof}

\begin{theorem}
	\label{aproximatesubdiferential}Assume\ that the linear growth condition 
	\cref{growth} holds with $x_{0}^{\ast }\in \mathnormal{L}_{w^{\ast
	}}^{\infty }(T,X^{\ast })$, and assume that either $X$ is Banach, or the $%
	f_{t}$'s are epi-pointed and $x_0^{\ast }(t)\in \inte \dom  f_{t}$ ae.
	Then $x^{\ast }\in \partial I_{f}(x)$ if and only if there exist nets of
	measurable functions $x_{i}: T \to X,$ $(x_{i}^{\ast })\subset \mathnormal{%
		L}_{w^{\ast }}^{\infty }(T,X^{\ast })$ such that $x_{i}^{\ast }(t)\in
	\partial f(t,x_{i}(t))$ ae, and
	
	\begin{enumerate}[label={(\alph*)},ref={(\alph*)}]
		
		\item \label{aproximatesubdiferentiala}$x^{\ast }={w^{\ast }}\mathnormal{-}%
		\lim \int\limits_{T}x_{i}^{\ast }(t)d\mu (t)$.
		
		\item \label{aproximatesubdiferentialb}$x_{i}\rightarrow x$ uniformly ae.
		
		\item \label{aproximatesubdiferentialc} $\displaystyle\int\limits_T |f(t, x_i(t)) -\langle x_i^*(t),  x_i(t)-x_0 \rangle -  f(t,x_0) |d\mu(t) \to 0$.
	\end{enumerate}
	
	If $X$ is reflexive, then the above nets are replaced by sequences, and the $%
	w^{\ast }$-convergence by norm-convergence.
\end{theorem}

\begin{proof}
	Let $u_{0}^{\ast }\in \partial I_{f}(x_{0})$ for $x_{0}\in X,$ $n\geq 0$, $%
	\rho _{X}$ a continuous seminorm in $X$ and $\rho _{X^{\ast }}$ a $w^{\ast }$%
	-continuous seminorm in $X^{\ast }.$ We choose $\epsilon \in (0,1/2n)$ such
	that $\epsilon \sup_{y^{\ast }\in \mathbb{B}_{\rho _{X}}^{\circ }(0,1)}\rho
	_{X^{\ast }}(y^{\ast })\leq 1/(2n)$. Then, by  \cref{teorema3.9} (or
	Lemma \ref{lemma:previo}, when $f_{t}$ are epi-pointed), we can choose $\ell
	\in \mathcal{I}(\varepsilon ^{2})$ and $z^{\ast }\in
	L_{w^{\ast }}^{1}(T,X^{\ast }),$ such that $\rho _{X^{\ast }}(u_{0}^{\ast
	}-\int_{T}z^{\ast }d\mu )\leq 1/(2n)$ and $z^{\ast }(t)\in \partial _{\ell
		(t)}f(t,x_{0})$ ae (when $f_{t}$ are epi-pointed we can take $z^{\ast }(t)\in \partial _{\ell (t)}f(t,x_{0})\cap \inte%
	( \dom  f_{t})$). We define the
	measurable multifunction $B:T\rightarrow X\times X^{\ast }$ as $(x,x^{\ast
	})\in B(t)$ if and only if
	
	\begin{enumerate}[label={B(\roman*)},ref={B(\roman*)}]
		
		\item\label{apxa}$x^*\in \sub f(t,x)$
		\item\label{apxb}$\rho_X(x-x_0)\leq  \epsilon$.
		\item\label{apxc}$x^* - z^*(t) \in \frac{\ell(t)}{\epsilon}B^\circ_{\rho_{X}}(0,1)$.
		\item\label{apxd}$|f(t, x) -\langle x^*,  x-x_0 \rangle -  f(t,x_0)|\leq 2\ell(t)$.
	\end{enumerate}
	By Br{\o}nsted-Rockafellar's theorem in the Banach case (see \cite[Theorem 1]%
	{MR705231}), and\ by \cite[Theorem 4.2]{BronstedRock} in the
	epi-pointed case, $B(t)$ is nonempty ae $t\in T$, and taking into account
	\cite[Theorem III.22.]{MR0467310}, there exists a measurable selection $(x(t), x^*(t)) \in B(t)$ such that 
	 $x^{\ast }\in \mathnormal{L}_{w^{\ast }}^{1}(T,X^{\ast })$ (by \cref%
	{apxc}, or  in $L^1(T,X)$, if $ z^* \in  L^1(T,X)$), such that  $\rho_{X}(x_0 - x(\cdot))_{\infty} \leq 1/n$,  
	\begin{align}
		\rho_{X^*}(x^*_0 - \int_T x^*(t)d\mu(t) &\leq \rho_{X^*}(x^*_0 - \int_T z^*(t)f\mu(t)) +\rho_{X^*}(x^*_0 - \int_T z^*(t)f\mu(t) \\
		&\leq 1/2n + \sup\limits_{y^*\in B^\circ_{\rho_{X}}(0,1)}\rho_{X^*}(y^*)\int_T \frac{\ell(t)}{\epsilon} d\mu(t)\\ 
		&\leq 1/2n +1/2n=1/n,
	\end{align}
	and $\displaystyle\int\limits_T |f(t, x_i(t)) -\langle x_i^*(t),  x_i(t)-x_0 \rangle -  f(t,x_0) |d\mu(t)\leq 1/n$.

	To prove  the sufficiency, let $x^*\in X^*$, $x\in X$, and nets $x_i(\cdot)$, $x^*_i(\cdot)$ that satisfy the conclusion of the Theorem. Then for every $y\in X$
	\begin{align*}
		\langle x^*_0, y-x_0\rangle & =  \langle x^*_0 - \displaystyle \int x^*_i, y-x_0\rangle + \displaystyle \int \langle x^*_i(t), y-x_i(t)\rangle +\displaystyle \int \langle x^*_i(t), x_i(t)-x_0\rangle\\
		&\leq  \langle x^*_0 - \displaystyle \int x^*_i, y-x_0\rangle + \displaystyle \int f(t,y) - \displaystyle \int f(t,x_i(t)) +\displaystyle \int \langle x^*_i(t), x_i(t)-x_0\rangle\\
		&\leq \langle x^*_0 - \displaystyle \int x^*_i, y-x_0\rangle + I_f(y) -I_f(x) \\&+ \displaystyle\int\limits_T |f(t, x_i(t)) -\langle x_i^*(t),  x_i(t)-x_0 \rangle -  f(t,x_0) |d\mu(t).
	\end{align*}
	So, taking the limits we conclude the result.
\end{proof}

\section{Conclusions}

\

We provide new calculus rules for the subdifferential and the $\varepsilon-$subdifferential of convex integral functions, by means only of the corresponding subdifferential of the data functions defining the integrand functions. This goal was achieved under appropriate and natural qualifiacation conditions that rely on continuity properties of the convex integrand.

\biboptions{sort,compress}
\bibliography{references}
\end{document}